\documentclass[a4paper]{scrartcl}
\addtokomafont{disposition}{\rmfamily}

\title{Learning and Leveraging Anisotropy Parameters in ANOVA Approximation}
\author{
    Felix~Bartel\thanks{Mathematisch-Geographische Fakultät, KU Eichstätt-Ingolstadt, 85270 Eichstät, Germany.\newline
    E-Mail: \texttt{felix.bartel@ku.de}}
    \and Pascal~Schröter\thanks{Faculty of Mathematics, Chemnitz University of Technology, 09107 Chemnitz, Germany.\newline
    E-Mail: \texttt{pascal.schroeter@math.tu-chemnitz.de}}
}

\usepackage[T1]{fontenc}
\usepackage[utf8]{inputenc}
\usepackage[english]{babel,csquotes}
\usepackage{hyperref}
\usepackage{xcolor}
\usepackage{setspace}
\usepackage{tabularx}
\usepackage{algorithm, algorithmic}
\usepackage[backend=biber,sorting=nyt,giveninits=true,maxbibnames=9]{biblatex}

\usepackage{xpatch}

\makeatletter
\patchcmd\blx@bblinput{\blx@blxinit}
                      {\blx@blxinit
                      }{}{\fail}
\makeatother

\usepackage{todonotes}

\usepackage{mathtools, amssymb, amsthm}
\usepackage{cleveref, thmtools}
 \usepackage{bm, dsfont} 

\numberwithin{equation}{section}

\newtheorem{theorem}{Theorem}[section]

\newtheorem{lemma}[theorem]{Lemma}

\newtheorem{defi}[theorem]{Definition}

\renewcommand{\tilde}{\widetilde}

\DeclareMathOperator{\Span}{span}
\DeclareMathOperator{\CV}{CV}
\DeclareMathOperator{\FCV}{FCV}
\DeclareMathOperator{\Supp}{supp}
\DeclareMathOperator{\Argmin}{arg\,min}

\begin{document} 

\maketitle
\begin{abstract}
    We present a Fourier-based approach for high-dimensional function approximation.
    To this end, we analyze the truncated ANOVA (analysis of variance) decomposition and learn the anisotropic smoothness properties of the target function from scattered data.
    This smoothness information is then incorporated into our approximation algorithm to improve the accuracy.
    Specifically, we employ least squares approximation using trigonometric polynomials in combination with frequency boxes of optimized aspect ratios.
    These frequency boxes allow for the application of the Nonequispaced Fast Fourier Transform (NFFT), which significantly accelerates the computation of the method.
    Our approach enables the efficient optimization of dozens of parameters to achieve high approximation accuracy with minimal overhead.
    Numerical experiments demonstrate the practical effectiveness of the proposed method.

      \end{abstract}

\section{Introduction} 

Scattered data approximation methods are typically designed with a specific class of functions in mind.
Any available prior information about the function to be approximated may be leveraged to fine-tune the approximation scheme, improving the accuracy.
In this work, we focus on high-dimensional functions with anisotropic smoothness properties, i.e., functions whose smoothness may vary significantly across different directions in the input space.

Rather than presuming knowledge of such smoothness properties a priori, we introduce a novel data-driven approach to learn the anisotropic smoothness directly from scattered function samples.
This learned information is then used to adapt the approximation procedure accordingly, yielding a substantial improvement in accuracy.
Moreover, since a better approximation enables a more precise estimation of the smoothness parameters, we embed this process into an iterative refinement loop.
This approach allows for the simultaneous estimation and tuning of dozens of smoothness-related parameters, which is performed efficiently and integrated directly into the approximation process.
Crucially, this enhancement only brings a small computational overhead to the algorithm, ensuring that the method remains scalable and computationally efficient.

To validate our approach, we conduct numerical experiments in $d=2,5,9$ dimensional space.
The code for the presented method integrated into the \texttt{ANOVAapprox.jl} software available on GitHub as an \texttt{julia} package.
The results demonstrate not only the significant practical improvement in approximation quality but also align with our theoretical analysis.

Existing work includes:
\begin{itemize}
\item
    The detection of anisotropic smoothness parameters via the (infinitely many) coefficients of a hyperbolic wavelet basis, which was investigated in \cite{SUV21}.
    The key part is that the wavelet basis is universal in that there is no need for adaption of the basis or a priori knowledge on the anisotropy.
    However, this method does not apply when only samples of a function are given.
\item
    In \cite{HJL18} adaptive cubature based on ``steady decay of Fourier coefficients'' was investigated, which introduces a stopping criterion for sampling from rank-1 lattices and digital nets based on a predetermined target accuracy.
    The approach we present in this paper is not adaptive, in that the samples are given but rather fixed to begin with.
\end{itemize}

For our purposes we use the trigonometric ANOVA decomposition, which is well-established for high-dimensional approximation, see e.g.~\cite{PS21,BPS22,PS22}.
The parameter of this method is the frequency index set $\mathcal I\subset\mathds Z^d$, which is a union of differently sized and shaped (axis parallel) boxes, cf.~\Cref{fig:anova_frequencies}.
\begin{figure} \centering
    \begingroup
  \makeatletter
  \providecommand\color[2][]{\GenericError{(gnuplot) \space\space\space\@spaces}{Package color not loaded in conjunction with
      terminal option `colourtext'}{See the gnuplot documentation for explanation.}{Either use 'blacktext' in gnuplot or load the package
      color.sty in LaTeX.}\renewcommand\color[2][]{}}\providecommand\includegraphics[2][]{\GenericError{(gnuplot) \space\space\space\@spaces}{Package graphicx or graphics not loaded}{See the gnuplot documentation for explanation.}{The gnuplot epslatex terminal needs graphicx.sty or graphics.sty.}\renewcommand\includegraphics[2][]{}}\providecommand\rotatebox[2]{#2}\@ifundefined{ifGPcolor}{\newif\ifGPcolor
    \GPcolortrue
  }{}\@ifundefined{ifGPblacktext}{\newif\ifGPblacktext
    \GPblacktexttrue
  }{}\let\gplgaddtomacro\g@addto@macro
\gdef\gplbacktext{}\gdef\gplfronttext{}\makeatother
  \ifGPblacktext
\def\colorrgb#1{}\def\colorgray#1{}\else
\ifGPcolor
      \def\colorrgb#1{\color[rgb]{#1}}\def\colorgray#1{\color[gray]{#1}}\expandafter\def\csname LTw\endcsname{\color{white}}\expandafter\def\csname LTb\endcsname{\color{black}}\expandafter\def\csname LTa\endcsname{\color{black}}\expandafter\def\csname LT0\endcsname{\color[rgb]{1,0,0}}\expandafter\def\csname LT1\endcsname{\color[rgb]{0,1,0}}\expandafter\def\csname LT2\endcsname{\color[rgb]{0,0,1}}\expandafter\def\csname LT3\endcsname{\color[rgb]{1,0,1}}\expandafter\def\csname LT4\endcsname{\color[rgb]{0,1,1}}\expandafter\def\csname LT5\endcsname{\color[rgb]{1,1,0}}\expandafter\def\csname LT6\endcsname{\color[rgb]{0,0,0}}\expandafter\def\csname LT7\endcsname{\color[rgb]{1,0.3,0}}\expandafter\def\csname LT8\endcsname{\color[rgb]{0.5,0.5,0.5}}\else
\def\colorrgb#1{\color{black}}\def\colorgray#1{\color[gray]{#1}}\expandafter\def\csname LTw\endcsname{\color{white}}\expandafter\def\csname LTb\endcsname{\color{black}}\expandafter\def\csname LTa\endcsname{\color{black}}\expandafter\def\csname LT0\endcsname{\color{black}}\expandafter\def\csname LT1\endcsname{\color{black}}\expandafter\def\csname LT2\endcsname{\color{black}}\expandafter\def\csname LT3\endcsname{\color{black}}\expandafter\def\csname LT4\endcsname{\color{black}}\expandafter\def\csname LT5\endcsname{\color{black}}\expandafter\def\csname LT6\endcsname{\color{black}}\expandafter\def\csname LT7\endcsname{\color{black}}\expandafter\def\csname LT8\endcsname{\color{black}}\fi
  \fi
    \setlength{\unitlength}{0.0500bp}\ifx\gptboxheight\undefined \newlength{\gptboxheight}\newlength{\gptboxwidth}\newsavebox{\gptboxtext}\fi \setlength{\fboxrule}{0.5pt}\setlength{\fboxsep}{1pt}\definecolor{tbcol}{rgb}{1,1,1}\begin{picture}(7920.00,3960.00)\gplgaddtomacro\gplbacktext{}\gplgaddtomacro\gplfronttext{}\gplbacktext
    \put(0,0){\includegraphics[width={396.00bp},height={198.00bp}]{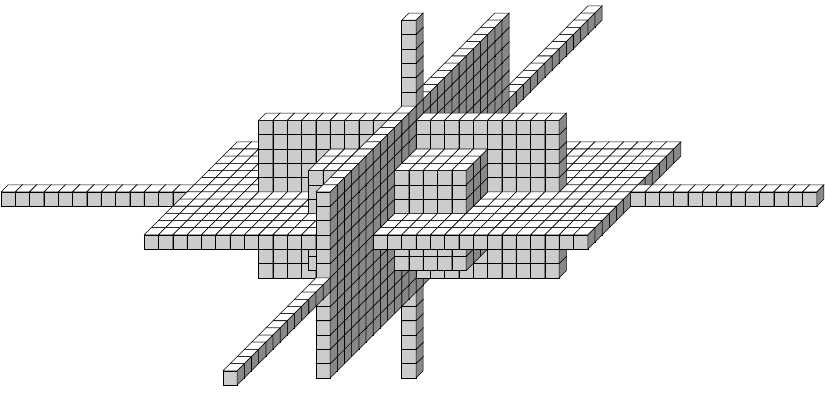}}\gplfronttext
  \end{picture}\endgroup
     \caption{Example frequencies in dimension $d=3$ used in ANOVA approximation with $12$ bandwidth parameters.}\label{fig:anova_frequencies}
\end{figure} This allows for the use of fast Fourier techniques implemented in the \texttt{GroupedTransforms.jl} package \cite{BPS22}, utilizing the Nonequispaced Fast Fourier Transform (\texttt{NFFT}) \cite{KKP09}.
So far cubes of frequencies with equal side length have been used with brute force or manual, heuristic choices for their size.
We automate this choice, including frequency boxes with non-equal side lengths by distributing the frequency budget such that it optimizes the error decay.
To model this we use a map constructing an increasing sequence of frequency index sets
\begin{equation*}
    \Psi\colon\mathds N\to\mathcal P(\mathds Z^d), m\mapsto\mathcal I
    \quad\text{such that}\quad
    |\Psi(m)| = m \,.
\end{equation*}
When every ANOVA term has a certain anisotropic Sobolev smoothness, see \Cref{sec:anisosobolev}, we observe a polynomial error decay $\|f-P_{\Psi(m)}f\|_{L_2} \lesssim m^{-s_{\Psi}}$, where the degree $s_{\Psi}$ depends on the given smoothness and the chosen $\Psi$.
Our goal is to choose $\Psi$ such that we have optimal approximation properties, i.e., the error decay $s_{\Psi}$ is maximal.

\paragraph{Structure of the paper.}
We start by introducing the ANOVA approximation, anisotropic Sobolev spaces, and the cross-validation score in \Cref{sec:pre}.
In \Cref{sec:la} we develop a method in order to learn the smoothness properties of a given function from samples, which we then use in \Cref{sec:ua} to improve the approximation accuracy.
We end with three numerical experiments in \Cref{sec:numerics} and some fineal remarks in \Cref{sec:conclusion}.

\paragraph{Notation.}
In this paper  we write $A_n\lesssim B_n$ or $B_n\gtrsim A_n$ if there exists $C>0$ such that $A_n \le C B_n$ for all $n\in\mathds N$;
when both relations hold we write $A_n\sim B_n$;
$\langle\cdot,\cdot\rangle$ is the Euclidean inner product;
$[d] \coloneqq \{1,\dots,d\}$;
$\mathds N$ are the natural numbers, $\mathds N_0 \coloneqq \mathds N\cup\{0\}$, and $2\mathds N_0$ are all even non-negative integers;
$\mathds T \coloneqq\mathds R/\mathds N$ is the one-dimensional torus.

 \section{Preliminaries}\label{sec:pre} 

\subsection{ANOVA approximation}\label{sec:anova} 

The analysis of variance (ANOVA) has its origin in statistics with the goal of identifying dimension interactions of multivariate, high-dimensional functions.
We only give a brief introduction with the domain restriction being the $d$-dimensional torus $\mathds T^d$ for simplicity, while in-depth literature and extensions can be found in e.g.~\cite{CMO97, RA99, LO06, NW08, KSWW09, Gu13, Schmischke22}.
The core idea is that certain functions are representable as a sum of lower-dimensional functions, like
\begin{equation*}
    f(x_1,\dots,x_9)
    =
    \Big|x_1-\frac 12\Big|
    +\cos(2\pi x_1)\cos(2\pi x_2) +\frac{\sin(2\pi x_3)}{2 + \sin(2\pi x_4)\sin(2\pi x_5)} \,.
\end{equation*}
The above function $f$ is nine-dimensional but may be decomposed into a sum of one one-dimensional function, one two-dimensional, and one three-dimensional one with five variables not even occurring.
This assumption occurs, e.g., naturally in calculations of the electronic structure problem for molecules in \cite{GHH11} where component-wise interactions are intrinsic.
Even when this assumption is not given, the truncation to lower-dimensional terms has been proven to beat past methods in practice on benchmark problems, cf.~\cite[Chapter~6]{Schmischke22}.

A central tool for the analysis are integral projections.
Let $\mathfrak u \subseteq [d]$ be a subset of coordinate indices and $\mathfrak u^\complement = [d]\setminus \mathfrak u$ its complement.
Further, for vectors $\bm x\in\mathds T^d$ indexed with a subset $\mathfrak u\subseteq[d]$ we define $\bm x_{\mathfrak u} \coloneqq (x_j)_{j\in \mathfrak u}$.
The \emph{integral projection} of $f$ with respect to $\mathfrak u\subseteq[d]$ is then given by
\begin{equation*}
    P_{\mathfrak u} f(\bm x)
    \coloneqq \int_{\mathds T^{d-|\mathfrak u|}} f(\bm x) \;\mathrm d\bm x_{\mathfrak u^\complement}\, .
\end{equation*}
For $\mathfrak u\subseteq [d]$, the \emph{ANOVA terms} and \emph{ANOVA decomposition} are given by
\begin{equation*}
    f_{\mathfrak u}
    = P_{\mathfrak u} f - \sum_{\mathfrak v\subseteq \mathfrak u} f_{\mathfrak v}
    \quad\text{and}\quad
    f = \sum_{\mathfrak u\subseteq [d]} f_{\mathfrak u} \,.
\end{equation*}
This orthogonal decomposition connects to a decomposition in Fourier space, where it divides the Fourier coefficients into disjoint sets of frequencies depending on their support $\Supp\bm k \coloneqq \{j\in [d] : k_j \neq 0\}$:
\begin{equation*}
    f_{\mathfrak u}
    = \sum_{\substack{\bm k\in\mathds Z^d \\ \Supp \bm k = \mathfrak u}}
    \hat f_{\bm k}
    \exp(2\pi\mathrm i\langle\bm k, \cdot\rangle) \,,
\end{equation*}
with $\hat f_{\bm k} = \langle f, \exp(2\pi\mathrm i\langle\bm k, \cdot\rangle) \rangle_{L_2} = \int_{\mathds T^d} f(\bm x)\exp(-2\pi\mathrm i\langle\bm k, \bm x\rangle)\;\mathrm d\bm x$.
For other domains and orthonormal systems this works analogously, see e.g.~\cite{Schmischke22}, or more involved with wavelets in \cite{LPU21}.
The number of ANOVA terms is $2^d$ and therefore grows exponentially in the dimension, which reflects the well-known curse of dimensionality.
The idea to circumvent this is to truncate the decomposition and only take a certain number of terms into account.
It is common to truncate to lower-dimensional terms $f_{\mathfrak u}$ with $|\mathfrak u| \le d_s$, with $d_s$ being called superposition dimension.
Then, the number of terms with respect to the spatial dimension $d$ is $\sum_{j=1}^{d_s} \binom{d_s}{d} \in \mathcal O(d^{d_s})$, which grows polynomially in $d$ instead of exponentially.
Further, among the terms $f_{\mathfrak u}$ it is possible to find the ones contributing most to the overall function via sensitivity analysis, reducing the number of terms even more, which is basically comparing the normalized $L_2$ norms of the ANOVA terms, cf.~\cite{BPS22,PS22b}.

In order to compute an approximation from samples, we truncate the Fourier series of each ANOVA term.
To this end we define for a bandwidth vector $\bm m = (m_1, \dots, m_d)\in 2\mathds N_0$
\begin{equation}\label{eq:anovafreqs1}
    \tilde{\mathcal I}_{\bm m}
    \coloneqq
    \bigtimes_{j=1}^{d} 
    \begin{cases}
        \{0\} & \text{if } m_j=0 \,, \\
        [-m_j/2, m_j/2) \cap \mathds Z \setminus \{0\} & \text{otherwise,} \phantom]
    \end{cases}
\end{equation}
which is a box of frequencies for the ANOVA term $f_{\mathfrak u}$ with $\mathfrak u = \Supp\bm m$.
The final frequency index set for the overall ANOVA approximation then becomes
\begin{equation}\label{eq:anovafreqs2}
    \mathcal I = \bigcup_{\mathfrak u\in U} \tilde{\mathcal I}_{\bm m_{\mathfrak u}} \,,
\end{equation}
with an example depicted in \Cref{fig:anova_frequencies}.

Given points $\bm X = \{\bm x^1, \dots, \bm x^n\}\subseteq\mathds T^d$ and samples $\bm y = [y_1,\dots,y_n]^\top\in\mathds C^n$, we define the \emph{least squares ANOVA approximation}
\begin{equation}\label{eq:lsqr}
    S_{\mathcal I}^{\bm X} \bm y
    \coloneqq \Argmin\Big\{
        \sum_{i=1}^{n} |g(\bm x^i)-y_i|^2 :
        g\in\Span\{\exp(2\pi\mathrm i\langle\bm k,\cdot\rangle)\}_{\bm k\in\mathcal I}
    \Big\} \,.
\end{equation}
Given the full rank of the system matrix $\bm L \coloneqq [ \exp(2\pi\mathrm i\langle\bm k,\bm x^i\rangle) ]_{i\in [n], \bm k\in \mathcal I} \in\mathds C^{n\times|\mathcal I|}$, the Fourier coefficients of the approximation are computable by solving a system of equations
\begin{equation*}
    S_{\mathcal I}^{\bm X}\bm y = \sum_{\bm k\in\mathcal I} \hat g_{\bm k}\exp(2\pi\mathrm i\langle\bm k,\cdot\rangle)
    \quad\text{with}\quad
    \bm{\hat g}
    = [\hat g_{\bm k}]_{\bm k\in\mathcal I}
    = (\bm L^\ast\bm L)^{-1}\bm L^\ast\bm y \,.
\end{equation*}
We solve that system with the iterative \texttt{LSQR} method \cite{PS82}, using only matrix-vector products.
With uniformly random points and at least logarithmic oversampling $n\ge 10 |\mathcal I|(\log|\mathcal I|+t)$, $t>0$ we know to have with probability exceeding $1-2\exp(-t)$ the condition number $\sigma_{\max}^2(\bm L)/\sigma_{\min}^2(\bm L) \le 3$, cf.~\cite[Lemmata~6.2 and 6.4]{Barteldiss}.
With that well-conditioned system matrix $\bm L$, the solution of a system of equations up to machine precision $\text{\texttt{eps}} = 10^{-16}$ requires at most $56$ iterations, cf.~\cite[Theorem~3.1.1]{Gre97}.
In our numerical experiments the maximal number of iterations does not exceed $25$.
Thus, the overall computational cost of the approximation is governed by a constant multiple of the computational cost of one matrix-vector product.

Because of the box structure in the frequencies, the Nonequispaced Fast Fourier Transform (\texttt{NFFT}, cf.~\cite{KKP09}) is applicable, making the approximation algorithm fast and parallelizable to a nearly arbitrary extent, cf.~\cite{BPS22, Schmischke22}.
For $n$ points and accuracy $\varepsilon$ for the matrix-vector product, this yields a computational cost of
\begin{equation*}
    \mathcal O\Big(\sum_{\mathfrak u\in U} \Big( \prod_{j\in \mathfrak u} m_{\mathfrak u,j} \Big) \log \Big( \prod_{j\in \mathfrak u} m_{\mathfrak u,j} \Big) + (\log\varepsilon)^{d_s}n\Big)
    = \mathcal O\Big(|\mathcal I|\log|\mathcal I| + (\log\varepsilon)^{d_s}n\Big)
\end{equation*}
in an FFT-like fashion.
The naive matrix-vector product would yield a computational cost and memory requirements of $\mathcal O(|\mathcal I|n)$.

The error of the ANOVA approximation decomposes into the individual ANOVA terms as well.

\begin{lemma}\label{anovaerror} The error of the ANOVA approximation $S_{\mathcal I}^{\bm X}\bm y$ \eqref{eq:lsqr} with $\mathcal I = \bigcup_{\mathfrak u\in U} \tilde{\mathcal I}_{\bm m_{\mathfrak u}}$ splits into the error of the approximation of the individual ANOVA terms $f_{\mathfrak u}$, i.e.,
    \begin{equation*}
        \|S_{\mathcal I}^{\bm X}\bm y - f\|_{L_2}^2
        = \sum_{\mathfrak u\in U} \|P_{\tilde{\mathcal I}_{\bm m_{\mathfrak u}}} S_{\mathcal I}^{\bm X}\bm y - f_{\mathfrak u}\|_{L_2}^{2} + \sum_{\mathfrak u\in \mathcal P([d])\setminus U} \|f_{\mathfrak u}\|_{L_2}^2 \,.
    \end{equation*}
\end{lemma} 

\begin{proof} The proof follows from the orthogonality of the ANOVA decomposition.
\end{proof} 

Thus, to describe the overall error behavior, it suffices to investigate the individual ANOVA terms.

\subsection{Anisotropic Sobolev spaces}\label{sec:anisosobolev} 

We model the individual ANOVA terms to be in \emph{anisotropic Sobolev spaces}
\begin{equation*}
    H^{s_1, \dots, s_d}
    \coloneqq \Big\{
    f\in L_2 :
    \frac{\partial^{s_1}}{\partial x_1^{s_1}}f, \dots, \frac{\partial^{s_d}}{\partial x_d^{s_d}}f \in L_2
    \Big\} \,,
\end{equation*}
with smoothness parameters $s_1, \dots, s_d\in\mathds N_0$.
These spaces capture different smoothness properties for different dimensions, see e.g.~\cite{Nikolskii75}.
Truncating the frequencies to boxes comes naturally with these spaces, justifying the use of the \texttt{NFFT}.

\begin{lemma}\label{lemma:Fouriercharacter} For $s_1,\dots,s_d\in\mathds N_0$ we have $f\in H^{s_1,\dots,s_d}$ if and only if
    \begin{equation}\label{eq:anisotropicnorm}
        \|f\|_{H^{s_1,\dots,s_d}}^2
        \coloneqq \sum_{\bm k\in\mathds Z^d} \max \{1, |k_1|^{2s_1}, \dots, |k_d|^{2s_d}\} | \hat f_{\bm k} |^2
        < \infty \,.
    \end{equation}
\end{lemma} 

\begin{proof} For the derivative of trigonometric polynomials we have $\partial^{s_j}/(\partial x_j^{s_j}) \exp(2\pi\mathrm i\langle\bm k, \cdot\rangle) = (2\pi\mathrm i k_j)^{s_j} \exp(2\pi\mathrm i\langle\bm k, \cdot\rangle)$.
    Thus, we have the following Fourier sum for the derivatives of a given function $f = \sum_{\bm k\in\mathds Z^d}\hat f_{\bm k}\exp(2\pi\mathrm i\langle\bm k,\cdot\rangle)$
    \begin{equation*}
        \frac{\partial^{s_j}}{\partial x_j^{s_j}} f
        = \sum_{\bm k\in\mathds Z^d} (2\pi\mathrm i k_j)^{s_j} \hat f_{\bm k} \exp(2\pi\mathrm i\langle\bm k,\cdot\rangle) \,.
    \end{equation*}
    By Parseval's identity it is immediate that $f\in H^{s_1,\dots,s_d}$, given the stated Fourier coefficient decay.
    For the reverse direction we have
    \begin{align*}
        \|f\|_{H^{s_1,\dots,s_d}}^2
        &= \sum_{\bm k\in\mathds Z^d} \max \{1, |k_1|^{2s_1}, \dots, |k_d|^{2s_d}\} | \hat f_{\bm k} |^2 \\
        &\le \sum_{\bm k\in\mathds Z^d} (1 + |k_1|^{2s_1} + |k_d|^{2s_d}) | \hat f_{\bm k} |^2 \\
        &= \sum_{\bm k\in\mathds Z^d} | \hat f_{\bm k} |^2 + \sum_{j=1}^{d} \sum_{\bm k\in\mathds Z^d} |k_j|^{2 s_j} | \hat f_{\bm k} |^2 \,,
    \end{align*}
    where all sums are finite due to $f\in H^{s_1,\dots,s_d}$ and Parseval's identity.
\end{proof} 

Note, with similar arguments we have that $f\in H^{s_1,\dots,s_d}$ automatically implies the presumably stronger condition $\partial^{\|\bm\alpha\|_1}/(\partial x_1^{\alpha_1} \dots \partial x_d^{\alpha_d}) f \in L_2$ for all $\bm\alpha = [\alpha_1,\dots,\alpha_d] \in \mathds N_0^d$ such that $\alpha_1/s_1+\dots+\alpha_d/s_d \le 1$.

From the characterization in terms of the decay of the Fourier coefficients, we immediately obtain the generalization of the anisotropic Sobolev spaces to non-integer smoothness  by using the norm \eqref{eq:anisotropicnorm}.
Note that because of the equivalence of $\ell_p$ (quasi)-norms, one could do the same for $\ell_p$-balls, cf.\ \cite[Section~3.6.1]{Barteldiss}.

Knowing the decay in the Fourier coefficients, we are able to investigate how the truncated Fourier sum behaves.

\begin{lemma}\label{boxprojection} Let $\bm m = (m_1, \dots, m_d)\in(2\mathds N)^d$ be a bandwidth vector and $s_1,\dots,s_d>0$ smoothness parameters.
    When projecting functions from anisotropic Sobolev spaces $H^{s_1,\dots,s_d}$ to nonempty frequency boxes
    \begin{equation*}
        \mathcal I_{\bm m}
        \coloneqq \bigtimes_{j=1}^{d}  [-m_j/2, m_j/2) \cap \mathds Z
    \end{equation*}
    we obtain
    \begin{equation*}
        \sup_{\|f\|_{H^{s_1, \dots, s_d}}\le 1} \|f-P_{\mathcal I_{\bm m}}f\|_{L_2}^{2}
        = \max \Big\{\Big(\frac{m_1}{2}\Big)^{-2s_1}, \dots, \Big(\frac{m_d}{2}\Big)^{-2s_d}\Big\} \,.
    \end{equation*}
\end{lemma} 

\begin{proof} For the upper bound, we use
    \begin{align*}
        \|f-P_{\mathcal I_{\bm m}}f\|_{L_2}^{2}
        &= \sum_{\bm k\notin \mathcal I_{\bm m}} |\hat f_{\bm k}|^2
        = \sum_{\bm k\notin \mathcal I_{\bm m}} (\max\{1,k_1^{s_1}, \dots, k_d^{s_d}\})^{-2} |\max\{1,k_1^{s_1}, \dots, k_d^{s_d}\}\hat f_{\bm k}|^2 \\
        &\le \|f\|_{H^{s_1, \dots, s_d}}^2 \sup_{\bm k\notin \mathcal I_{\bm m}} (\max\{k_1^{s_1}, \dots, k_d^{s_d}\})^{-2} \\
        &= \|f\|_{H^{s_1, \dots, s_d}}^2 \Big(\inf_{\bm k\notin \mathcal I_{\bm m}} \max\{k_1^{s_1}, \dots, k_d^{s_d}\}\Big)^{-2} \,,
    \end{align*}
    which evaluates to the assertion due to the definition of $\mathcal I_{\bm m}$.

    For the lower bound, we construct a fooling function consisting of a trigonometric monomial
    \begin{equation*}
        g = \max\{\ell_1^{-s_1}, \dots, \ell_d^{-s_d}\}\exp(2\pi\mathrm \langle\bm \ell,\cdot\rangle)
        \quad\text{for}\quad
        \bm \ell\in\Argmin_{\bm k\notin \mathcal I_{\bm m}}\{\min\{k_1^{s_1},\dots,k_d^{s_d}\}\} \,.
    \end{equation*}
    This function has an $H^{s_1,\dots,s_d}$ norm of one, and it holds
    \begin{equation*}
        \sup_{\|f\|_{H^{s_1, \dots, s_d}}\le 1} \|f-P_{\mathcal I_{\bm m}}f\|_{L_2}^{2}
        \ge \|g-P_{\mathcal I_{\bm m}}g\|_{L_2}^{2}
        = \sup_{\bm k\notin \mathcal I_{\bm m}} \max\{k_1^{-2s_1}, \dots, k_d^{-2s_d}\} \,. \qedhere
    \end{equation*}
\end{proof} 

\Cref{boxprojection} shows the advantage of using frequency boxes instead of cubes when approximating in anisotropic Sobolev spaces.
When we have a frequency budget of $|\mathcal I_{\bm m}| = m\in\mathds N$ approximating with frequency cubes with side length $m_j = \sqrt[d]{m}$ for $j=1,\dots,d$ yields
\begin{equation*}
    \sup_{\|f\|_{H^{s_1,\dots,s_d}}\le 1} \|f-P_{\mathcal I_{\bm m}} f\|_{L_2}^{2}
    \sim m^{-2\min\{s_1, \dots, s_d\}/d} \,,
\end{equation*}
whereas the optimal box ratio $m_j = (m^{1/(1/s_1+\dots+1/s_d)})^{1/s_j}$ for $j=1,\dots,d$ gives
\begin{equation*}
    \sup_{\|f\|_{H^{s_1,\dots,s_d}}\le 1} \|f-P_{\mathcal I_{\bm m}} f\|_{L_2}^{2}
    \sim m^{-2/(1/s_1+\dots+1/s_d)} \,.
\end{equation*}
For $d=2$, $s_1 = 1$, and $s_2 = 3$, this would make a difference of $m^{-1}$ in contrast to $m^{-3/2}$ for the optimal box ratio, which is the core motivation for this paper.

\subsection{Fast cross-validation}\label{sec:fcv} 

The central question of this paper is to choose parameters such that the approximation has a small prediction error.
It is therefore crucial to have a fast and reliable estimator of the $L_2$ error.
A basic idea is to split the data into a training set and a validation set for estimating the error.
Doing this multiple times, we obtain a reasonable estimator for the $L_2$ error functional known as cross-validation score, which is widely used in learning, see e.g.,~\cite{TaWe96, BS02, MS00, Rosset09, DPR10}.
A special case is where the partitionings seclude single points, then the training sets become $\{(\bm x^1, y_1), \dots, (\bm x^{i-1}, y_{i-1}), (\bm x^{i+1}, y_{i+1}), \dots, (\bm x^n, y_n)\}\subseteq\mathds T^d\times\mathds C$ and the validation sets $\{(\bm x^i, y_i)\}\subseteq\mathds T^d\times\mathds C$.
This leads to the so-called \emph{leave-one-out cross-validation score}.

\begin{defi} Let $S_{\mathcal I}^{\bm X}\bm y\colon\mathds T^d\to\mathds C$ be an approximation based on the data samples $\{(\bm x^1, y_1), \dots, (\bm x^n, y_n)\} \subseteq \mathds T^d\times \mathds C$.
    Further, let $S_{\mathcal I}^{\bm X_{-i}}\bm y_{-i}\colon\mathds T^d\to\mathds C$ be the same method applied to the samples with the $i$-th sample omitted.
    The \emph{cross-validation score} is defined via
    \begin{align*}
        \CV(S_{\mathcal I}^{\bm X}\bm y)
        = \frac 1 n \sum_{i=1}^{n} \Big|\Big(S_{\mathcal I}^{\bm X_{-i}}{\bm y_{-i}}\Big)(\bm x^i)-y_i \Big|^2 \,.
    \end{align*}
\end{defi} 

This parameter choice strategy is used widely in practice, and theoretical validation for the least squares approximation was shown in \cite[Corollary~9.11]{Barteldiss}.

A drawback of the cross-validation score is the numerical complexity of having to compute the $n$ approximations $S_{\mathcal I}^{\bm X_{-i}}\bm y_{-i}$ for $i=1,\dots,n$.
To circumvent this, the \emph{approximated cross-validation score} of the least squares approximation $S_{\mathcal I}^{\bm X}\bm y$ was introduced in \cite{BHP20} via
\begin{equation}\label{eq:fcv}
    \FCV(S_{\mathcal I}^{\bm X}\bm y)
    = \frac 1n \sum_{i=1}^n \frac{|(S_{\mathcal I}^{\bm X}\bm y)(\bm x^i)-y_i|^2}{(1-|\mathcal I|/n)^2} \,.
\end{equation}
It was shown that this is the same as the actual cross-validation score $\CV(S_{\mathcal I}^{\bm X}\bm y)$ for exact quadrature points, cf.\ \cite{BHP20}.
If we do not have exact quadrature like with the scattered data setting assumed in this paper, it is still an excellent approximation and can be used instead, cf.\ \cite[Theorem~9.20]{Barteldiss}.
 \section{Learning anisotropy from the ANOVA approximation}\label{sec:la} 

In this section we propose a method to estimate the smoothness parameters of a function $f$ based on samples.
We model every ANOVA term coming from an anisotropic Sobolev space, i.e., $f_{\mathfrak u}\in H^{\bm s_{\mathfrak u}}$ with smoothness parameters $\bm s_{\mathfrak u} = [s_{\mathfrak u,j}]_{j\in \mathfrak u}$.
The natural choice of frequencies is then a union of frequency boxes $\mathcal I = \bigcup_{\mathfrak u\in U}\tilde{\mathcal I}_{\bm m_{\mathfrak u}}$ with $\tilde{\mathcal I}_{\bm m_{\mathfrak u}}$ from \eqref{eq:anovafreqs1}.
This makes the discussed fast Fourier methods in the software package \texttt{ANOVAapprox.jl} introduced in \Cref{sec:anova} applicable.
Further, the truncation error splits into its ANOVA components
\begin{equation*}
    \| f - P_{\mathcal I} f \|_{L_2}^2
    = \sum_{\mathfrak u\in U} \|f_{\mathfrak u} - P_{\tilde{\mathcal I}_{\bm m_{\mathfrak u}}} f\|_{L_2}^2 + \sum_{\mathfrak u\in \mathcal P([d])\setminus U} \|f_{\mathfrak u}\|_{L_2}^2 \,.
\end{equation*}
With a reasonable choice of $U$, the second sum becomes small.
We estimate the first sum by the worst-case error and obtain with \Cref{boxprojection}
\begin{equation}\label{eq:behavior}
    \| f - P_{\mathcal I} f \|_{L_2}^2
    \le \sum_{\mathfrak u\in U} 
        \max \Big\{\Big(\frac{m_{\mathfrak u,j}}{2}\Big)^{-2s_{\mathfrak u,j}}\Big\}_{j\in \mathfrak u} \|f_{\mathfrak u}\|_{H^{\bm s_{\mathfrak u}}}^2
    + \sum_{\mathfrak u\in \mathcal P([d])\setminus U} \|f_{\mathfrak u}\|_{L_2}^2\, .
\end{equation}

Our goal is to extract the smoothness parameters $s_{\mathfrak u,j}$ in order to adapt the bandwidth parameters $m_{\mathfrak u,j}$ defining $\mathcal I$ and controlling the approximation error.
Without loss of generality, we aim to estimate the smoothness parameter $s_{\mathfrak u,\mathfrak u_1}$ of the ANOVA term $\mathfrak u$.
In order to do so we use projections, where we vary the bandwidth in that specific dimension of that ANOVA term.
For small bandwidths $(m_{\mathfrak u,\mathfrak u_1}/2)^{s_{\mathfrak u,\mathfrak u_1}}\le \min \{(m_{\mathfrak u,j}/2)^{s_{\mathfrak u,j}}\}_{j\in \{\mathfrak u_2,\dots,\mathfrak u_{|\mathfrak u|}\}}$ the error is then dominated by that dimension $\mathfrak u_1$ and we have
\begin{equation}\label{eq:kittchen}
    \|f_{\mathfrak u}-P_{\tilde{\mathcal I}_{\bm m_{\mathfrak u}}} f\|_{L_2}^2
    \le \max \Big\{\Big(\frac{m_{\mathfrak u,j}}{2}\Big)^{-2s_{\mathfrak u,j}}\Big\}_{j\in \mathfrak u} \|f_{\mathfrak u}\|_{H^{\bm s_{\mathfrak u}}}^2
    = \Big(\frac{m_{\mathfrak u,\mathfrak u_1}}{2}\Big)^{-2s_{\mathfrak u,\mathfrak u_1}} \|f_{\mathfrak u}\|_{H^{\bm s_{\mathfrak u}}}^2 \,.
\end{equation}
For $(m_{\mathfrak u,\mathfrak u_1}/2)^{s_{\mathfrak u,\mathfrak u_1}}\ge \min \{(m_{\mathfrak u,j}/2)^{s_{\mathfrak u,j}}\}_{j\in \mathfrak u\setminus\{\mathfrak u_1\}}$ the error then flattens.
We use the first range in order to extract the decay $s_{\mathfrak u,\mathfrak u_1}$.

For now this uses the $L_2$-projection, which is not available to us.
With more information -- like Wavelet coefficients -- this was already investigated in \cite{SUV21}.
We have approximated Fourier coefficients from the least squares ANOVA approximation.
The error of the $L_2$-projection to a frequency set $\mathcal I_{(\bm m_{\mathfrak u})'}$ is equal to the $2$-norm of all Fourier coefficients of the tail outside $\mathcal I_{(\bm m_{\mathfrak u})'}$, i.e.
\begin{equation*}
    \|f-P_{\mathcal I_{(\bm m_{\mathfrak u})'} } f\|_{L_2}^2
    = \sum_{\bm k\notin \mathcal I_{(\bm m_{\mathfrak u})'} } |\hat f_{\bm k}|^2 \,.
\end{equation*}
The least squares ANOVA approximation \eqref{eq:lsqr} works with a finite frequency index set $\mathcal I_{\bm m_{\mathfrak u}}$ to begin with and gives only an estimate of the exact Fourier coefficients.
By taking the $2$-norm of the approximated Fourier coefficients in $\mathcal I_{\bm m_{\mathfrak u}}\setminus \mathcal I_{(\bm m_{\mathfrak u})'}$ this gives a reasonable estimate, as the following lemma shows.

\begin{lemma}\label{lemma:v1v2} Let $f\colon \mathds T^d\to\mathds C$ be a function and $g\in W$ an approximation thereof from a function space $W$.
    Further, let $W = V_1 \oplus V_2$ and $g = g_1 + g_2$ with $g_1\in V_1$ and $g_2\in V_2$.
    
    Then
    \begin{equation*}
        \|g_2\|_{L_2} - \|f-g\|_{L_2}
        \le \|f-P_{V_1} f\|_{L_2}
        \le \|g_2\|_{L_2} + \|f-g\|_{L_2}
    \end{equation*}
\end{lemma} 

\begin{proof} We obtain the left-hand inequality using
    \begin{equation*}
        \|g_2\|_{L_2}
        \le \|g_2 - P_{V_2} f\|_{L_2} + \|P_{V_2} f\|_{L_2}
        \le \|g - f\|_{L_2} + \|f-P_{V_1} f\|_{L_2} \,.
    \end{equation*}
    The right-hand inequality follows from
    \begin{equation*}
        \|f-P_{V_1} f\|_{L_2}
        \le \|f-g_1\|_{L_2}
        = \|f-g+g_2\|_{L_2}
        \le \|f-g\|_{L_2} + \|g_2\|_{L_2} \,. \qedhere
    \end{equation*}
\end{proof} 

We apply this by choosing $g = S_{\mathcal I}^{\bm X}\bm y$ the least squares ANOVA approximation.
For extracting the smoothness in dimension $j$ of the ANOVA term in dimensions $\mathfrak v$ we use the decomposition $g_1 = P_{\mathcal I'(\mathfrak v,j,m)} S_{\mathcal I}^{\bm X}\bm y$ and $g_2 = P_{\mathcal I\setminus\mathcal I'(\mathfrak v,j,m)} S_{\mathcal I}^{\bm X}\bm y$ with
\begin{equation}\label{eq:Ip}
    \mathcal I'(\mathfrak v,j,m')
    \coloneqq \tilde{\mathcal I}_{(m_{\mathfrak v,1}, \dots, m_{\mathfrak v,j-1}, m', m_{\mathfrak v,j+1}, \dots, m_{\mathfrak v,d})} \cup \bigcup_{\mathfrak u\in U\setminus\{v\}} \tilde{\mathcal I}_{\bm m_{\mathfrak u}} \,.
\end{equation}
This yields a vector
\begin{equation*}
    \Big[ \|P_{\mathcal I\setminus\mathcal I'(v,j,m')} S_{\mathcal I}^{\bm X}\bm y\|_{L_2}^2 \Big]_{m'\in\{0,\dots,m_{\mathfrak u,j}\}} \,,
\end{equation*}
which contains the sought smoothness decay, which eventually flattens for large $m'$, as explained in \eqref{eq:kittchen}.
In order to extract the smoothness information, we have to identify the $\bar m$ where the flattening begins in order to estimate the smoothness from the components $0,\dots,\bar m$.

With an initial guess, the frequency boxes are likely not optimal, and many of the exact Fourier coefficients have a smaller magnitude than the truncation error.
In the approximation this error spreads evenly among the coefficients of $S_{\mathcal I}^{\bm X}\bm y$.
In particular, the part of the function in $\Span\{\exp(2\pi\mathrm i\langle\bm k,\cdot\rangle)\}_{\bm k\in\mathcal I}$ will be reconstructed, and the remainder resembles the approximation of noise either from the measurement process or the truncation itself, for which the even spread is quantified in the following lemma:

\begin{lemma} Let $\mathcal I\subseteq\mathds Z^d$ be a frequency index set, $t>0$, $\bm X$ be i.i.d.\ uniformly random points with $|\bm X|\ge 10|\mathcal I|(\log|\mathcal I| + t)$ for the points $\bm X$, and $\bm\varepsilon\in\mathds C^n$ be i.i.d.\ mean-zero, random noise with variance $\sigma^2$.
    The expected magnitude of the approximated Fourier coefficients $S_{\mathcal I}^{\bm X}\bm\varepsilon = \sum_{\bm k\in\mathcal I} \hat g_{\bm k}\exp(2\pi\mathrm i\langle\bm k,\cdot\rangle)$ then equals
    \begin{equation*}
        \frac{2\sigma^2}{3n}
        \le
        \mathds E_{\bm\varepsilon}(|\hat g_{\bm k}|^2)
        \le
        \frac{2\sigma^2}{n} \,.
    \end{equation*}
    with probability $1-2\exp(-t)$ in the random choice of points.
\end{lemma} 

\begin{proof} Applying the least squares approximation \eqref{eq:lsqr} to i.i.d.\ noise $\bm\varepsilon$ with variance $\sigma ^2$ gives the approximated Fourier coefficients $ \bm{\hat g} = (\bm L^\ast\bm L)^{-1}\bm L^\ast\bm\varepsilon $.
    Thus,
    \begin{align*}
        \mathds E_{\bm\varepsilon}(|\hat g_k|^2)
        &= \mathds E_{\bm\varepsilon}\Big(\Big| \sum_{i=1}^{n} [(\bm L^\ast\bm L)^{-1}\bm L^\ast]_{i,k} \varepsilon_i \Big|^2\Big) \\
        &= \sum_{i=1}^{n} \sum_{j=1}^{n} [(\bm L^\ast\bm L)^{-1}\bm L^\ast]_{i,k} [\bm L(\bm L^\ast\bm L)^{-1}]_{k,j} \mathds E_{\bm\varepsilon}( \varepsilon_i \overline{\varepsilon_j}) \\
        &= \sigma^2 \sum_{i=1}^{n} [(\bm L^\ast\bm L)^{-1}\bm L^\ast]_{i,k} [\bm L(\bm L^\ast\bm L)^{-1}]_{k,i} \\
        &= \sigma^2 [(\bm L^\ast\bm L)^{-1}\bm L^\ast\bm L(\bm L^\ast\bm L)^{-1}]_{k,k} \\
        &= \sigma^2 [(\bm L^\ast\bm L)^{-1}]_{k,k} \,.
    \end{align*}

    To estimate the diagonal entries $[(\bm L^\ast\bm L)^{-1}]_{k,k}$ we use \cite[Lemmata~6.2 and 6.4]{Barteldiss}, which gives
    \begin{equation*}
        \frac{n}{2}
        \le
        \lambda_{\min}( \bm L^\ast\bm L )
        \le \lambda_{\max}( \bm L^\ast\bm L )
        \le \frac{3}{2n} \,,
    \end{equation*}
    with the stated probability $1-2\exp(-t)$.
    This is equivalent to the Rayleigh--Ritz quotient satisfying
    \begin{equation*}
        \frac{2n}{3}
        \le
        \frac{ \bm x^\ast (\bm L^\ast\bm L)^{-1} \bm x}{\bm x^\ast\bm x}
        \le
        \frac{2}{n}
        \quad\text{for all}\quad
        \bm x\in\mathds C^{|\mathcal I|} \,.
    \end{equation*}
    In particular, we have for $\bm x = \bm e_k$
    \begin{equation*}
        \frac{2n}{3}
        \le
        [\bm L^\ast\bm L]_{k,k}^{-1}
        \le
        \frac{2}{n} \,. \qedhere
    \end{equation*}
\end{proof} 

Thus, the flat plane corresponds to the most common value $c$ in the magnitude of all approximated Fourier coefficients.
For each $\mathfrak u$ and $j\in[|\mathfrak u|]$, we set $\bar m$ the largest $m$ such that
\begin{equation*}
    \Big[\|P_{\mathcal I\setminus\mathcal I'(\mathfrak v,j\mathfrak ,m)} S_{\mathcal I}^{\bm X}\bm y\|_{L_2}^2 \Big]_m
    > c^2 |\mathcal I\setminus\mathcal I'(\mathfrak v,j\mathfrak ,m)| \,.
\end{equation*}

It remains to estimate the smoothness from $[\|P_{\mathcal I\setminus\mathcal I'(\mathfrak v,j,m)} S_{\mathcal I}^{\bm X}\bm y\|_{L_2}^2 ]_{m\in\{0,\dots,\bar m\}}$.
For that we use weighted linear least squares in the log-log scale.

\begin{theorem}\label{learrates} Let $C_1,C_2,s>0$ and $C_1i^{-2s}\le y_i \le C_2i^{-2s}$ for $i=1,\dots,n$ modeling the error being in a tube with slope $-2s$.
    Applying weighted linear least squares in the log-log scale with weights $\omega_i = 1/(H_n i)$, where $H_n$ is the $n$-th harmonic number and points $x_i = \log i$ yields the approximated decay behavior $D i^{-2t}$ with
    \begin{equation*}
        D = \exp\Big(\frac{(\sum_{i=1}^{n}\omega_i \log^2(i))(\sum_{i=1}^{n} \omega_i \log(y_i)) - (\sum_{i=1}^{n} \omega_i \log(i) \log(y_i))(\sum_{i=1}^{n}\omega_i \log(i))}{(\sum_{i=1}^{n}\omega_i \log^2(i))-(\sum_{i=1}^{n}\omega_i \log(i))^2}\Big) \\
    \end{equation*}
    and
    \begin{equation*}
        t = -\frac 12 \frac{(\sum_{i=1}^{n}\omega_i \log(i) \log(y_i)) - (\sum_{i=1}^{n} \omega_i \log(i))(\sum_{i=1}^{n} \omega_i \log(y_i))}{(\sum_{i=1}^{n}\omega_i \log^2(i)) - (\sum_{i=1}^{n}\omega_i \log(i))^2} \,.
    \end{equation*}
    If $n\ge 3$, the error for the smoothness parameters is bounded by
    \begin{equation*}
        |t-s| \,\le\, \frac{4\log(C_2/C_1)}{\log n}
    \end{equation*}
    and
    \begin{equation*}
        \log C_1 - 4\log\Big(\frac{C_2}{C_1}\Big)
        \;\le\; \log(D)
        \;\le\; \log C_2 + 4\log\Big(\frac{C_2}{C_1}\Big) \,.
    \end{equation*}
\end{theorem} 

\begin{proof} The solution of the weighted least squares is derived by computing the roots of the linear least squares functional
    \begin{equation*}
        \sum_{i=1}^{n} \omega_i | \log(y_i) - \log(D i^{-2t})|^2
        = \sum_{i=1}^{n} \omega_i | \log(y_i) - \log(D) -2t\log(i))|^2
    \end{equation*}
    using basic linear algebra.
    
    In order to prove the error estimates on $t$ and $D$, we first note
    $\sum_{i=1}^{n}\omega_i\log^2(i) \ge (\sum_{i=1}^{n}\omega_i\log(i))^2$.
    Thus,
    \begin{align}
        &s-t
        = s + \frac 12 \frac{(\sum_{i=1}^{n}\omega_i \log(i) \log(y_i)) - (\sum_{i=1}^{n} \omega_i \log(i))(\sum_{i=1}^{n} \omega_i \log(y_i))}{(\sum_{i=1}^{n}\omega_i \log^2(i)) - (\sum_{i=1}^{n}\omega_i \log(i))^2}\nonumber \\
        &\le s + \frac 12 \frac{(\sum_{i=1}^{n}\omega_i \log(i) ( \log(C_2) - 2s\log i )) - (\sum_{i=1}^{n} \omega_i \log(i))(\sum_{i=1}^{n} \omega_i (\log(C_1) - 2s\log i)}{(\sum_{i=1}^{n}\omega_i \log^2(i)) - (\sum_{i=1}^{n}\omega_i \log(i))^2}\nonumber \\
        &= s + \frac 12
        \Big(
        \log\Big(\frac{C_2}{C_1}\Big)\frac{\sum_{i=1}^{n}\omega_i \log(i)}{(\sum_{i=1}^{n}\omega_i \log^2(i)) - (\sum_{i=1}^{n}\omega_i \log(i))^2}
        -2s\Big)\nonumber \\
        &= \frac{\log(C_2/C_1)}{2}
        \frac{\sum_{i=1}^{n}\omega_i \log(i)}{(\sum_{i=1}^{n}\omega_i \log^2(i)) - (\sum_{i=1}^{n}\omega_i \log(i))^2}\label{eq:alligatoah} \,.
    \end{align}
    We obtain the same estimate for $t-s$ analogously.
    In order to estimate the latter fraction, we first need to estimate the sums by integrals taking their monotonicity into account
    \begin{equation}\label{eq:overkill1}
        \sum_{i=1}^{n} \frac{\log i}{i}
        \,\le\, \frac{\log 2}{2} + \frac{\log 3}{3} + \int_{3}^{n}\frac{\log x}{x}\;\mathrm dx
        \,=\, \frac{\log 2}{2} + \frac{\log 3}{3} + \frac{\log^2 n}{2} - \frac{\log^2 3}{2}
    \end{equation}
    and
    \begin{equation}\label{eq:overkill2}
        \frac{\log^3 n}{3} - \frac{\log^3 8}{3}
\sum_{i=1}^{n} \frac{\log^2 i}{i}
        \;\ge\;
        \frac{\log^3 (n+1)}{3} - \frac{\log^3 8}{3}
        + \sum_{i=1}^{7}\frac{\log^2 (i)}{i} \,.
    \end{equation}
    Using \eqref{eq:overkill1}, \eqref{eq:overkill2}, and $\log n\le H_n$ in \eqref{eq:alligatoah} we obtain for $n\ge 3$
    \begin{align*}
        &\frac{\sum_{i=1}^{n}\omega_i \log(i)}{(\sum_{i=1}^{n}\omega_i \log^2(i)) - (\sum_{i=1}^{n}\omega_i \log(i))^2}
        = \frac{\sum_{i=1}^{n}\log(i)/i}{(\sum_{i=1}^{n} \log^2(i)/i) - \frac{1}{H_n}(\sum_{i=1}^{n} \log(i)/i)^2} \\
        &\le \frac{\frac{\log^2 n}{2} - \frac{\log^2 3}{2} + \frac{\log 2}{2} + \frac{\log 3}{3}}{\frac{\log^3 (n+1)}{3} - \frac{\log^3 8}{3} + \sum_{i=2}^{7}\frac{\log^2 (i)}{i} - \frac{1}{\log n}(\frac{\log^2 n}{2} - \frac{\log^2 3}{2} + \frac{\log 2}{2} + \frac{\log 3}{3})^2}
        \le \frac{7}{\log n} \,,
    \end{align*}
    where the last inequality follows from simple analysis of the expression at hand.

    For the upper bound on $\log(D)$ we use
    \begin{align*}
        \log D
        &\le \frac{(\sum_{i=1}^{n}\omega_i \log^2 (i))(\log C_2 \sum_{i=1}^{n}\omega_i-2s\sum_{i=1}^{n}\omega_i\log (i))}{(\sum_{i=1}^{n}\omega_i\log^2 (i))-(\sum_{i=1}^{n}\omega_i \log (i))^2} \\
        &\phantom\le - \frac{(\log C_1\sum_{i=1}^{n}\omega_i\log (i) - 2s\sum_{i=1}^{n}\omega_i\log^2 (i))(\sum_{i=1}^{n}\omega_i\log (i))}{(\sum_{i=1}^{n}\omega_i\log^2 (i))-(\sum_{i=1}^{n}\omega_i \log (i))^2} \\
        &= \frac{\log C_2 (\sum_{i=1}^{n}\omega_i\log^2 (i)) - \log C_1 (\sum_{i=1}^{n}\omega_i\log (i))^2}{(\sum_{i=1}^{n}\omega_i\log^2 (i))-(\sum_{i=1}^{n}\omega_i \log (i))^2} \\
        &= \log C_2 + \log\Big(\frac{C_2}{C_1}\Big) \frac{(\sum_{i=1}^{n}\omega_i \log (i))^2}{(\sum_{i=1}^{n}\omega_i\log^2 (i))-(\sum_{i=1}^{n}\omega_i \log (i))^2} \,.
    \end{align*}
    Analogously, we obtain for the lower bound
    \begin{equation*}
        \log D
        \ge
        \log C_1 + \log\Big(\frac{C_2}{C_1}\Big) \frac{(\sum_{i=1}^{n}\omega_i \log (i))^2}{(\sum_{i=1}^{n}\omega_i\log^2 (i))-(\sum_{i=1}^{n}\omega_i \log (i))^2} \,.
    \end{equation*}
    It remains to estimate the fraction for $n\ge 3$ by using \eqref{eq:overkill1}, \eqref{eq:overkill2}, and $\log n\le H_n$:
    \begin{align*}
        &\frac{(\sum_{i=1}^{n}\omega_i \log(i))^2}{(\sum_{i=1}^{n}\omega_i \log^2(i)) - (\sum_{i=1}^{n}\omega_i \log(i))^2}
        = \frac{(\sum_{i=1}^{n}\log(i)/i)^2}{H_n(\sum_{i=1}^{n} \log^2(i)/i) - (\sum_{i=1}^{n} \log(i)/i)^2} \\
        &\le \frac{(\frac{\log^2 n}{2} - \frac{\log^2 3}{2} + \frac{\log 2}{2} + \frac{\log 3}{3})^2}{\log n(\frac{\log^3 (n+1)}{3} - \frac{\log^3 8}{3} + \sum_{i=2}^{7}\frac{\log^2 (i)}{i}) - (\frac{\log^2 n}{2} - \frac{\log^2 3}{2} + \frac{\log 2}{2} + \frac{\log 3}{3})^2} 
        \le 4 \,,
    \end{align*}
    where the last inequality follows from simple analysis of the expression at hand.
\end{proof} 

We summarize our procedure in \Cref{algo:la}.

\noindent\begin{algorithm} \caption{\texttt{learning smoothness parameters}}\label{algo:la}
    \begin{tabularx}{\textwidth}{lp{140pt}X}
        \textbf{Input:} &
        $S_{\mathcal I}^{\bm X}\bm y$ & ANOVA approximation \\[2pt]
        \hline\\[-8pt]
        \textbf{Output:} &
        $J_{\mathfrak u}$ for $\mathfrak u\in U$ & sets of dimensions for which the smoothness estimation succeeded \\
        & $D_{\mathfrak u,j}$ and $s_{\mathfrak u,j}$ for $j\in\mathfrak u$, $\mathfrak u\in U$ & estimated smoothness parameters \\[2pt]\hline
    \end{tabularx}
    \begin{algorithmic}[1]
        \setstretch{1.1}
        \STATE{define $c$ to be the most common magnitude of the Fourier coefficients of $S_{\mathcal I}^{\bm X}\bm y$}
        \FOR{$\mathfrak u\in U$}
        \STATE{
        set $J_{\mathfrak u} \leftarrow \emptyset$
        }
        \FOR{$j\in\mathfrak u$}
        \STATE{
            find the largest $\bar m_{\mathfrak u,j}$ such that
            $ [\|P_{\mathcal I\setminus\mathcal I'(\mathfrak u,j\mathfrak ,m)} S_{\mathcal I}^{\bm X}\bm y\|_{L_2}^2 ]_m > c^2 |\mathcal I\setminus\mathcal I'(\mathfrak u,j\mathfrak ,m)| $
            for $m=0,\dots,\bar m$ with $\mathcal I'(\mathfrak u,j,m)$ as in \eqref{eq:Ip}
        }
        \IF{$\bar m_{\mathfrak u,j} \ge 3$}
        \STATE{
            define $\bm v_{\mathfrak u,j} \leftarrow [ \|P_{\mathcal I\setminus\mathcal I'(v,j,m')} S_{\mathcal I}^{\bm X}\bm y\|_{L_2}^2 ]_{m'\in\{0,\dots,\bar m_{\mathfrak u,j}\}}$
        }
        \STATE{
            compute $D_{\mathfrak u,j}$ and $s_{\mathfrak u,j}$ via weighted linear least squares in the log-log scale applied to $\bm v_{\mathfrak u,j}$ according to \Cref{learrates}
        }
        \STATE{
            set $J_{\mathfrak u}\leftarrow J_{\mathfrak u}\cup\{j\}$
        }
        \ENDIF
        \ENDFOR
        \ENDFOR
        \RETURN{$J_{\mathfrak u}$, $D_{\mathfrak u,j}$, and $s_{\mathfrak u,j}$}
    \end{algorithmic}
\end{algorithm}

 \section{Using anisotropy in ANOVA approximation}\label{sec:ua} 

In this section we use the estimated smoothness parameters $D_{\mathfrak u,j}$ and $s_{\mathfrak u,j}$ from \Cref{sec:la} in order to compute a new set of frequencies $\psi(m) = \mathcal I$, improving the approximation quality.
As it may happen that we are not able to detect the smoothness parameters for certain dimensions for a lack of available data, we define the set $J_{\mathfrak u}$, which collects all dimensions $j\in\mathfrak u$ for which the smoothness parameter estimation was successful, cf.~\Cref{algo:la}.

According to \cite[Theorem~1.1]{Bartel22}, given logarithmic oversampling $n\ge 10|\mathcal I|(\log|\mathcal I| + t)$, the error of the least squares ANOVA approximation $S_{\mathcal I}^{\bm X}\bm y$ is bounded by
\begin{equation}\label{eq:sap}
    \|f-S_{\mathcal I}^{\bm X}\bm y\|_{L_2}^{2}
    \lesssim \|f-P_{\mathcal I}\bm y\|_{L_2}^{2} + \sigma^2\frac{|\mathcal I|}{n} \,.
\end{equation}
with probability exceeding $1-3\exp(-t)$.
These two summands resemble
\begin{itemize}
    \item the truncation error behaving the same as the truncated Fourier sum $\|P_{\mathcal I} f - f\|_{L_2}$, which we already know from \eqref{eq:behavior}, and
    \item the error due to noise, which only depends increasingly on the number of frequencies and not their shape.
\end{itemize}
Finding a good frequency shape $\psi$ with a fixed frequency budget $m\in\mathds N$ for the least squares ANOVA approximation $S_{\psi(m)}^{\bm X}\bm y$ is therefore the same as finding good frequencies for the truncated Fourier sum $P_{\psi(m)}f$ of which we know the error behavior \eqref{eq:behavior}.
Thus, we are able to compute the optimal bandwidths by solving the optimization problem
\begin{equation}\label{eq:opti}
    \begin{aligned}
        \min_{m_{\mathfrak u,j}} \quad &\sum_{\mathfrak u\in U} \max_{j\in J_{\mathfrak u}} C_{\mathfrak u,j} (m_{\mathfrak u,j}-1)^{-2s_{\mathfrak u,j}} \\
        \mathrm{s.t.} \quad &\sum_{\mathfrak u\in U} \prod_{j=1}^{|\mathfrak u|} (m_{\mathfrak u,j}-1) = m-1 \,.
    \end{aligned}
\end{equation}

\begin{lemma}\label{optimalbandwidth} Let $d\in\mathds N$ be the dimension, $m\in\mathds N$ the frequency budget, $U\subseteq\mathcal P([d])$ the active ANOVA terms, and $J_{\mathfrak u} \subseteq [|\mathfrak u|]$ for $\mathfrak u\in U$ the ANOVA terms for which we have smoothness parameters.
    Further, let $C_{\mathfrak u,j}>0$, $s_{\mathfrak u,j}>0$ for $j\in J_{\mathfrak u}$, and $m_{\mathfrak u,j}>0$ for $j\in \mathfrak u\setminus J_{\mathfrak u}$.
    Then the solution of \eqref{eq:opti} is given by computing $\lambda>0$ such that the following monotone equation is fulfilled
    \begin{equation}\label{eq:optimallambda}
        \sum_{\mathfrak u\in U} B_{\mathfrak u}^{\frac{1}{1+A_{\mathfrak u}}} (\lambda A_{\mathfrak u})^{-\frac{A_{\mathfrak u}}{1+A_{\mathfrak u}}}
        = m-1 \,,
    \end{equation}
    with
    \begin{equation*}
        A_{\mathfrak u} \coloneqq \frac 12 \sum_{j\in J_{\mathfrak u}} \frac{1}{s_{\mathfrak u,j}}
        \quad\text{and}\quad
        B_{\mathfrak u} \coloneqq \prod_{j\in J_{\mathfrak u}} C_{\mathfrak u,j}^{\frac{1}{2s_{\mathfrak u,j}}} \prod_{j\in[|\mathfrak u|]\setminus J_{\mathfrak u}} m_{\mathfrak u,j}-1 \,.
    \end{equation*}
    Finally, we obtain the bandwidths optimizing \eqref{eq:opti} via
    \begin{equation*}
        m_{\mathfrak u,j} = \Big(\frac{C_{\mathfrak u,j}}{(\lambda B_{\mathfrak u} A_{\mathfrak u})^{\frac{1}{1+A_{\mathfrak u}}}}\Big)^{\frac{1}{2s_{\mathfrak u,j}}} + 1 \,.
    \end{equation*}
\end{lemma} 

\begin{proof} Note that the individual terms in the inner $\max$ of the optimization problem can be assumed equal, as otherwise there are bandwidths $m_{\mathfrak u,j}$ yielding a smaller target value whilst satisfying the constraint.
    With the substitution
    \begin{equation*}
        z_{\mathfrak u} = C_{\mathfrak u,j} (m_{\mathfrak u,j}-1)^{-2s_{\mathfrak u,j}}
        \quad\Leftrightarrow\quad
        m_{\mathfrak u,j}
        = \Big(\frac{z_{\mathfrak u}}{C_{\mathfrak u,j}}\Big)^{2s_{\mathfrak u,j}} + 1 \,,
    \end{equation*}
    the reduced optimization problem has the form
    \begin{align*}
        \min_{m_{\mathfrak u,j}} \quad &\sum_{\mathfrak u\in U} z_{\mathfrak u} \\
        \mathrm{s.t.} \quad &\sum_{\mathfrak u\in U} \prod_{j=1}^{|\mathfrak u|} \Big(\frac{z_{\mathfrak u}}{C_{\mathfrak u,j}}\Big)^{-\frac{1}{2s_{\mathfrak u,j}}}
        = \sum_{\mathfrak u\in U} B_{\mathfrak u} z_{\mathfrak u}^{-A_{\mathfrak u}}
        = m-1 \,.
    \end{align*}
    The solution is obtained by computing the roots of the Lagrangian
    \begin{flalign*}
        && \frac{\partial\mathcal L(z_{\mathfrak u},\lambda)}{\partial z_{\mathfrak u}}
        &= 1 - \lambda A_{\mathfrak u} B_{\mathfrak u} z_{\mathfrak u}^{-A_{\mathfrak u}-1}
        \overset != 0 && \\
        \Leftrightarrow &&
        z_{\mathfrak u} &= (\lambda A_{\mathfrak u} B_{\mathfrak u})^{\frac{1}{A_{\mathfrak u}+1}} \,. &&
    \end{flalign*}
    Plugging this into the constraint yields the defining equation for $\lambda$.
    With $\lambda$ computed, this gives $z_{\mathfrak u}$ and, then, $m_{\mathfrak u,j}$.
\end{proof} 

\begin{figure} \centering
    \begingroup
  \makeatletter
  \providecommand\color[2][]{\GenericError{(gnuplot) \space\space\space\@spaces}{Package color not loaded in conjunction with
      terminal option `colourtext'}{See the gnuplot documentation for explanation.}{Either use 'blacktext' in gnuplot or load the package
      color.sty in LaTeX.}\renewcommand\color[2][]{}}\providecommand\includegraphics[2][]{\GenericError{(gnuplot) \space\space\space\@spaces}{Package graphicx or graphics not loaded}{See the gnuplot documentation for explanation.}{The gnuplot epslatex terminal needs graphicx.sty or graphics.sty.}\renewcommand\includegraphics[2][]{}}\providecommand\rotatebox[2]{#2}\@ifundefined{ifGPcolor}{\newif\ifGPcolor
    \GPcolortrue
  }{}\@ifundefined{ifGPblacktext}{\newif\ifGPblacktext
    \GPblacktexttrue
  }{}\let\gplgaddtomacro\g@addto@macro
\gdef\gplbacktext{}\gdef\gplfronttext{}\makeatother
  \ifGPblacktext
\def\colorrgb#1{}\def\colorgray#1{}\else
\ifGPcolor
      \def\colorrgb#1{\color[rgb]{#1}}\def\colorgray#1{\color[gray]{#1}}\expandafter\def\csname LTw\endcsname{\color{white}}\expandafter\def\csname LTb\endcsname{\color{black}}\expandafter\def\csname LTa\endcsname{\color{black}}\expandafter\def\csname LT0\endcsname{\color[rgb]{1,0,0}}\expandafter\def\csname LT1\endcsname{\color[rgb]{0,1,0}}\expandafter\def\csname LT2\endcsname{\color[rgb]{0,0,1}}\expandafter\def\csname LT3\endcsname{\color[rgb]{1,0,1}}\expandafter\def\csname LT4\endcsname{\color[rgb]{0,1,1}}\expandafter\def\csname LT5\endcsname{\color[rgb]{1,1,0}}\expandafter\def\csname LT6\endcsname{\color[rgb]{0,0,0}}\expandafter\def\csname LT7\endcsname{\color[rgb]{1,0.3,0}}\expandafter\def\csname LT8\endcsname{\color[rgb]{0.5,0.5,0.5}}\else
\def\colorrgb#1{\color{black}}\def\colorgray#1{\color[gray]{#1}}\expandafter\def\csname LTw\endcsname{\color{white}}\expandafter\def\csname LTb\endcsname{\color{black}}\expandafter\def\csname LTa\endcsname{\color{black}}\expandafter\def\csname LT0\endcsname{\color{black}}\expandafter\def\csname LT1\endcsname{\color{black}}\expandafter\def\csname LT2\endcsname{\color{black}}\expandafter\def\csname LT3\endcsname{\color{black}}\expandafter\def\csname LT4\endcsname{\color{black}}\expandafter\def\csname LT5\endcsname{\color{black}}\expandafter\def\csname LT6\endcsname{\color{black}}\expandafter\def\csname LT7\endcsname{\color{black}}\expandafter\def\csname LT8\endcsname{\color{black}}\fi
  \fi
    \setlength{\unitlength}{0.0500bp}\ifx\gptboxheight\undefined \newlength{\gptboxheight}\newlength{\gptboxwidth}\newsavebox{\gptboxtext}\fi \setlength{\fboxrule}{0.5pt}\setlength{\fboxsep}{1pt}\definecolor{tbcol}{rgb}{1,1,1}\begin{picture}(4520.00,2540.00)\gplgaddtomacro\gplbacktext{\colorrgb{0.00,0.00,0.00}\put(994,575){\makebox(0,0)[r]{\strut{}\scriptsize $10^{-3}$}}\colorrgb{0.00,0.00,0.00}\put(994,1143){\makebox(0,0)[r]{\strut{}\scriptsize $10^{-2}$}}\colorrgb{0.00,0.00,0.00}\put(994,1711){\makebox(0,0)[r]{\strut{}\scriptsize $10^{-1}$}}\colorrgb{0.00,0.00,0.00}\put(994,2280){\makebox(0,0)[r]{\strut{}\scriptsize $10^{0}$}}\colorrgb{0.00,0.00,0.00}\put(1095,407){\makebox(0,0){\strut{}\scriptsize $10^{0}$}}\colorrgb{0.00,0.00,0.00}\put(2646,407){\makebox(0,0){\strut{}\scriptsize $10^{1}$}}\colorrgb{0.00,0.00,0.00}\put(4197,407){\makebox(0,0){\strut{}\scriptsize $10^{2}$}}}\gplgaddtomacro\gplfronttext{\csname LTb\endcsname \put(447,1427){\rotatebox{-270.00}{\makebox(0,0){\strut{}$L_2$ error}}}\csname LTb\endcsname \put(2646,167){\makebox(0,0){\strut{}$m$}}}\gplbacktext
    \put(0,0){\includegraphics[width={226.00bp},height={127.00bp}]{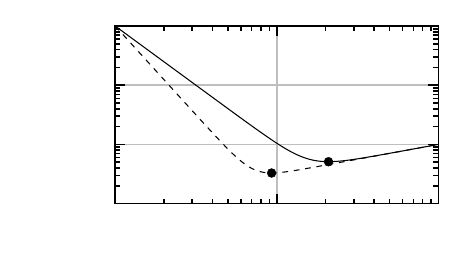}}\gplfronttext
  \end{picture}\endgroup
     \caption{Error behavior $\|S_{\Psi(m)}^{\bm X}\bm y - f\|_{L_2}^2 \lesssim m^{-2s_\Psi}+\sigma^2m/n$ for $s_\Psi=2$ (solid) and $s_\Psi=3$ (dashed) in the presence of noise.}
    \label{fig:error_behavior}
\end{figure} 

In order to implement \Cref{optimalbandwidth} we need to solve the nonlinear equation \eqref{eq:optimallambda}, which we do with bisection using the monotonicity.
Thus, having smoothness information and a frequency budget $m$, we are able to compute improved bandwidths.
In \Cref{sec:la} we covered how to estimate the smoothness, so it remains to choose the frequency budget.
For that we use the known error behavior from \eqref{eq:sap}.
Consequently, in the absence of noise, the frequency budget $|\mathcal I| = m$ should be chosen as large as possible while still satisfying the logarithmic oversampling condition.
When noise is present, one has to find $m$ such that over- and underfitting are balanced, i.e., the $L_2$ error is smallest.
Instead of minimizing the $L_2$ error, which is not available to use, we minimize the cross-validation score from \Cref{sec:fcv} in order to find the optimal frequency budget $m$.
The expected behavior and possible gain are depicted in \Cref{fig:error_behavior}.

 \section{Numerical results}\label{sec:numerics} 

In this section we test our approach with three different numerical examples.
For all of them we conduct two experiments.
\begin{itemize}
\item
    We sample the function exactly in $n=100\,000$ uniformly random points $\bm X$ and use a frequency budget $m$ such that we have logarithmic oversampling $m\log m = n$
    We initialize the smoothness parameters with $D_{\mathfrak u,j} = 1$ and $t_{\mathfrak u,j} = 1$ for all $j\in \mathfrak u$ and $\mathfrak u\in U$.
    This gives a frequency index set $\psi_1(m)$ for which we compute the first approximation $S_{\psi_1(m)}^{\bm X}\bm y$.
    From that approximation we estimate new smoothness parameters according to \Cref{algo:la}, which we use for a new frequency index set $\psi_{2}(m)$ and a new approximation $S_{\psi_2(m)}^{\bm X}\bm y$.
    We repeat this for $9$ iterations and approximate the $L_2$ error for every iteration using another set of $1\,000\,000$ uniformly random points.
\item
    In a second experiment we use noisy function values $\bm y = [ f(\bm x^i) + \varepsilon_i ]_{i=1}^{n}$ with Gaussian noise and a signal-to-noise ratio of
    \begin{equation*}
        \mathrm{SNR}_{\mathrm{dB}}
        = 10 \log_{10}\Big(\frac{\sum_{i=1}^{n}|f(\bm x^i)|^2}{\sum_{i=1}^{n}|\varepsilon_i|^2}\Big)
        = 50\,.
    \end{equation*}
    For the initial smoothness parameters $D_{\mathfrak u,j} = 1$ and $t_{\mathfrak u,j} = 1$ for all $j\in \mathfrak u$ and $\mathfrak u\in U$, we compute the fast cross-validation score, cf.~\Cref{sec:fcv}, and approximate the $L_2$ error for several values of $m\in\{300, \dots, 10\,000\}$.
    We choose $m$ such that it minimizes the fast cross-validation score $\FCV(S_{\mathcal I}^{\bm X}\bm y)$ defined in \eqref{eq:fcv} and estimate the smoothness parameters according to \Cref{algo:la}.
    We repeat this $3$ times.
\end{itemize}

Note that this setup gives plenty of output for academic evaluation.
For a practical implementation, the number of iterations from the first experiment could be reduced, and the cross-validation score of the second experiment would be used in conjunction with an optimization procedure to reduce the computation time further.
The corresponding code is integrated into the \texttt{ANOVAapprox.jl} software package.

\subsection{Example with complete ANOVA decomposition} 

The first example has spatial dimension $d=2$ with the function being
\begin{equation}\label{eq:2}
    f(\bm x) = \sqrt{\frac{378000}{2281}} \Big(p_2(x_1) + p_4(x_2) + p_4(x_1)p_2(x_2)\Big) \,,
\end{equation}
where the prefactor is such that $\|f\|_{L_2} = 1$ and $p_2$ and $p_4$ are the Bernoulli polynomials
\begin{equation*}
    p_2 = x^2-x+1/6
    \quad\text{and}\quad
    p_4 = x^4-2x^3+x^2-1/30 \,.
\end{equation*}
Their Fourier series is given by
\begin{equation*}
    p_n(x) = -\frac{n!}{(2\pi\mathrm i)^n}\sum_{k\neq 0} \frac{\exp(2\pi\mathrm i kx)}{k^n} \,.
\end{equation*}
Thus, $p_n$ has smoothness $s = n-1/2$.
With the zeroth Fourier coefficient zero, the ANOVA decomposition is immediately given by $f_{\{1\}}(x_1) = p_2(x_1)$, $f_{\{2\}}(x_2) = p_4(x_2)$, and $f_{\{1,2\}}(x_1,x_2) = p_4(x_1)p_2(x_2)$.
We use all ANOVA terms $\{1\}$, $\{2\}$, and $\{1,2\}$.

\begin{table}
    \centering
    \begin{tabular}{c|ccc}
        & $\mathfrak u = \{1\}$ & $\mathfrak u = \{2\}$ & $\mathfrak u = \{1,2\}$ \\
        \hline
        $j=1$ & $1.612$ ($1.5$) & & $3.958$ ($3.5$) \\
        $j=2$ & & $3.859$ ($3.5$) & $1.717$ ($1.5$)
    \end{tabular}
    \caption{Estimated rates $s_{\mathfrak u,j}$ for the $d=2$ example for every dimension of each ANOVA term $\mathfrak u$ with the analytical rates in brackets.}\label{s2_rates}
\end{table}

In the noiseless case, the estimated rates are close to the actual rates with overestimation throughout, cf.~\Cref{s2_rates}.
Notice that this example highlights that the ANOVA terms do not necessarily inherit smoothness among themselves but can behave entirely independently.
When it comes to the $L_2$ error in \Cref{s2_rates_l2error}, we see an improvement of a factor of $10$ with the first iteration, which does not change much in further iterations.
\begin{figure} \centering
    \begingroup
  \makeatletter
  \providecommand\color[2][]{\GenericError{(gnuplot) \space\space\space\@spaces}{Package color not loaded in conjunction with
      terminal option `colourtext'}{See the gnuplot documentation for explanation.}{Either use 'blacktext' in gnuplot or load the package
      color.sty in LaTeX.}\renewcommand\color[2][]{}}\providecommand\includegraphics[2][]{\GenericError{(gnuplot) \space\space\space\@spaces}{Package graphicx or graphics not loaded}{See the gnuplot documentation for explanation.}{The gnuplot epslatex terminal needs graphicx.sty or graphics.sty.}\renewcommand\includegraphics[2][]{}}\providecommand\rotatebox[2]{#2}\@ifundefined{ifGPcolor}{\newif\ifGPcolor
    \GPcolortrue
  }{}\@ifundefined{ifGPblacktext}{\newif\ifGPblacktext
    \GPblacktexttrue
  }{}\let\gplgaddtomacro\g@addto@macro
\gdef\gplbacktext{}\gdef\gplfronttext{}\makeatother
  \ifGPblacktext
\def\colorrgb#1{}\def\colorgray#1{}\else
\ifGPcolor
      \def\colorrgb#1{\color[rgb]{#1}}\def\colorgray#1{\color[gray]{#1}}\expandafter\def\csname LTw\endcsname{\color{white}}\expandafter\def\csname LTb\endcsname{\color{black}}\expandafter\def\csname LTa\endcsname{\color{black}}\expandafter\def\csname LT0\endcsname{\color[rgb]{1,0,0}}\expandafter\def\csname LT1\endcsname{\color[rgb]{0,1,0}}\expandafter\def\csname LT2\endcsname{\color[rgb]{0,0,1}}\expandafter\def\csname LT3\endcsname{\color[rgb]{1,0,1}}\expandafter\def\csname LT4\endcsname{\color[rgb]{0,1,1}}\expandafter\def\csname LT5\endcsname{\color[rgb]{1,1,0}}\expandafter\def\csname LT6\endcsname{\color[rgb]{0,0,0}}\expandafter\def\csname LT7\endcsname{\color[rgb]{1,0.3,0}}\expandafter\def\csname LT8\endcsname{\color[rgb]{0.5,0.5,0.5}}\else
\def\colorrgb#1{\color{black}}\def\colorgray#1{\color[gray]{#1}}\expandafter\def\csname LTw\endcsname{\color{white}}\expandafter\def\csname LTb\endcsname{\color{black}}\expandafter\def\csname LTa\endcsname{\color{black}}\expandafter\def\csname LT0\endcsname{\color{black}}\expandafter\def\csname LT1\endcsname{\color{black}}\expandafter\def\csname LT2\endcsname{\color{black}}\expandafter\def\csname LT3\endcsname{\color{black}}\expandafter\def\csname LT4\endcsname{\color{black}}\expandafter\def\csname LT5\endcsname{\color{black}}\expandafter\def\csname LT6\endcsname{\color{black}}\expandafter\def\csname LT7\endcsname{\color{black}}\expandafter\def\csname LT8\endcsname{\color{black}}\fi
  \fi
    \setlength{\unitlength}{0.0500bp}\ifx\gptboxheight\undefined \newlength{\gptboxheight}\newlength{\gptboxwidth}\newsavebox{\gptboxtext}\fi \setlength{\fboxrule}{0.5pt}\setlength{\fboxsep}{1pt}\definecolor{tbcol}{rgb}{1,1,1}\begin{picture}(4520.00,2540.00)\gplgaddtomacro\gplbacktext{\colorrgb{0.00,0.00,0.00}\put(1107,575){\makebox(0,0)[r]{\strut{}\scriptsize $1\cdot 10^{-6}$}}\colorrgb{0.00,0.00,0.00}\put(1107,1143){\makebox(0,0)[r]{\strut{}\scriptsize $1\cdot 10^{-5}$}}\colorrgb{0.00,0.00,0.00}\put(1107,1711){\makebox(0,0)[r]{\strut{}\scriptsize $1\cdot 10^{-4}$}}\colorrgb{0.00,0.00,0.00}\put(1107,2280){\makebox(0,0)[r]{\strut{}\scriptsize $1\cdot 10^{-3}$}}\colorrgb{0.00,0.00,0.00}\put(1208,407){\makebox(0,0){\strut{}\scriptsize $1$}}\colorrgb{0.00,0.00,0.00}\put(1582,407){\makebox(0,0){\strut{}\scriptsize $2$}}\colorrgb{0.00,0.00,0.00}\put(1955,407){\makebox(0,0){\strut{}\scriptsize $3$}}\colorrgb{0.00,0.00,0.00}\put(2329,407){\makebox(0,0){\strut{}\scriptsize $4$}}\colorrgb{0.00,0.00,0.00}\put(2703,407){\makebox(0,0){\strut{}\scriptsize $5$}}\colorrgb{0.00,0.00,0.00}\put(3076,407){\makebox(0,0){\strut{}\scriptsize $6$}}\colorrgb{0.00,0.00,0.00}\put(3450,407){\makebox(0,0){\strut{}\scriptsize $7$}}\colorrgb{0.00,0.00,0.00}\put(3824,407){\makebox(0,0){\strut{}\scriptsize $8$}}\colorrgb{0.00,0.00,0.00}\put(4197,407){\makebox(0,0){\strut{}\scriptsize $9$}}}\gplgaddtomacro\gplfronttext{\csname LTb\endcsname \put(257,1427){\rotatebox{-270.00}{\makebox(0,0){\strut{}$L_2$ error}}}\csname LTb\endcsname \put(2703,167){\makebox(0,0){\strut{}iteration}}}\gplbacktext
    \put(0,0){\includegraphics[width={226.00bp},height={127.00bp}]{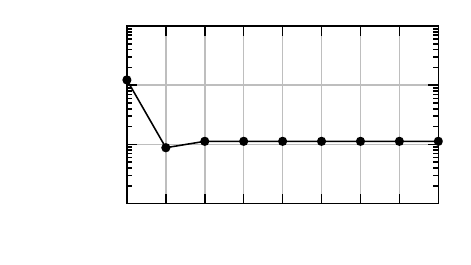}}\gplfronttext
  \end{picture}\endgroup
     \caption{$L_2$ error for the $d=2$ example.}\label{s2_rates_l2error}
\end{figure} In order to depict the frequency distribution, we have drawn boxes in \Cref{s2_rates_box} such that the area of the box represents the total amount of frequencies and each column corresponds to one ANOVA term with their width being the proportional number of frequencies.
Each column is then divided into rows for each occurring dimension in the ANOVA term, with the height being the proportional bandwidth.
\begin{figure} \centering
    \begingroup
  \makeatletter
  \providecommand\color[2][]{\GenericError{(gnuplot) \space\space\space\@spaces}{Package color not loaded in conjunction with
      terminal option `colourtext'}{See the gnuplot documentation for explanation.}{Either use 'blacktext' in gnuplot or load the package
      color.sty in LaTeX.}\renewcommand\color[2][]{}}\providecommand\includegraphics[2][]{\GenericError{(gnuplot) \space\space\space\@spaces}{Package graphicx or graphics not loaded}{See the gnuplot documentation for explanation.}{The gnuplot epslatex terminal needs graphicx.sty or graphics.sty.}\renewcommand\includegraphics[2][]{}}\providecommand\rotatebox[2]{#2}\@ifundefined{ifGPcolor}{\newif\ifGPcolor
    \GPcolortrue
  }{}\@ifundefined{ifGPblacktext}{\newif\ifGPblacktext
    \GPblacktexttrue
  }{}\let\gplgaddtomacro\g@addto@macro
\gdef\gplbacktext{}\gdef\gplfronttext{}\makeatother
  \ifGPblacktext
\def\colorrgb#1{}\def\colorgray#1{}\else
\ifGPcolor
      \def\colorrgb#1{\color[rgb]{#1}}\def\colorgray#1{\color[gray]{#1}}\expandafter\def\csname LTw\endcsname{\color{white}}\expandafter\def\csname LTb\endcsname{\color{black}}\expandafter\def\csname LTa\endcsname{\color{black}}\expandafter\def\csname LT0\endcsname{\color[rgb]{1,0,0}}\expandafter\def\csname LT1\endcsname{\color[rgb]{0,1,0}}\expandafter\def\csname LT2\endcsname{\color[rgb]{0,0,1}}\expandafter\def\csname LT3\endcsname{\color[rgb]{1,0,1}}\expandafter\def\csname LT4\endcsname{\color[rgb]{0,1,1}}\expandafter\def\csname LT5\endcsname{\color[rgb]{1,1,0}}\expandafter\def\csname LT6\endcsname{\color[rgb]{0,0,0}}\expandafter\def\csname LT7\endcsname{\color[rgb]{1,0.3,0}}\expandafter\def\csname LT8\endcsname{\color[rgb]{0.5,0.5,0.5}}\else
\def\colorrgb#1{\color{black}}\def\colorgray#1{\color[gray]{#1}}\expandafter\def\csname LTw\endcsname{\color{white}}\expandafter\def\csname LTb\endcsname{\color{black}}\expandafter\def\csname LTa\endcsname{\color{black}}\expandafter\def\csname LT0\endcsname{\color{black}}\expandafter\def\csname LT1\endcsname{\color{black}}\expandafter\def\csname LT2\endcsname{\color{black}}\expandafter\def\csname LT3\endcsname{\color{black}}\expandafter\def\csname LT4\endcsname{\color{black}}\expandafter\def\csname LT5\endcsname{\color{black}}\expandafter\def\csname LT6\endcsname{\color{black}}\expandafter\def\csname LT7\endcsname{\color{black}}\expandafter\def\csname LT8\endcsname{\color{black}}\fi
  \fi
    \setlength{\unitlength}{0.0500bp}\ifx\gptboxheight\undefined \newlength{\gptboxheight}\newlength{\gptboxwidth}\newsavebox{\gptboxtext}\fi \setlength{\fboxrule}{0.5pt}\setlength{\fboxsep}{1pt}\definecolor{tbcol}{rgb}{1,1,1}\begin{picture}(7920.00,2260.00)\gplgaddtomacro\gplbacktext{}\gplgaddtomacro\gplfronttext{\csname LTb\endcsname \put(549,940){\makebox(0,0){\strut{}\scriptsize 1}}\csname LTb\endcsname \put(706,940){\makebox(0,0){\strut{}\scriptsize 2}}\csname LTb\endcsname \put(2131,470){\makebox(0,0){\strut{}\scriptsize 1}}\csname LTb\endcsname \put(2131,1410){\makebox(0,0){\strut{}\scriptsize 2}}\csname LTb\endcsname \put(1974,2096){\makebox(0,0){\strut{}initial frequency distribution}}}\gplgaddtomacro\gplbacktext{}\gplgaddtomacro\gplfronttext{\csname LTb\endcsname \put(5232,940){\makebox(0,0){\strut{}\scriptsize 1}}\csname LTb\endcsname \put(6744,31){\makebox(0,0){\strut{}\scriptsize 1}}\csname LTb\endcsname \put(6744,971){\makebox(0,0){\strut{}\scriptsize 2}}\csname LTb\endcsname \put(5924,2096){\makebox(0,0){\strut{}final frequency distribution}}}\gplbacktext
    \put(0,0){\includegraphics[width={396.00bp},height={113.00bp}]{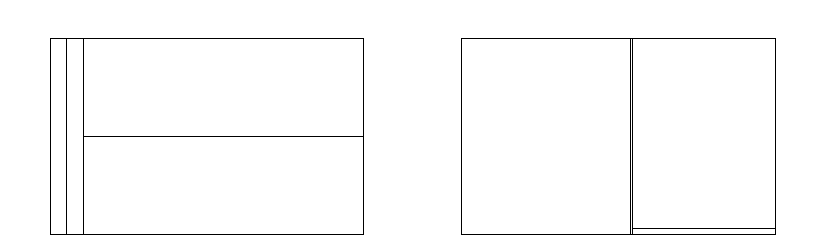}}\gplfronttext
  \end{picture}\endgroup
     \caption{Depiction of the frequency distribution in iteration $1$ and $9$ in the ANOVA terms for the $d=2$ example.}\label{s2_rates_box}
\end{figure} We observe that a lot more of the frequency budget was spent on the ANOVA term $\{1\}$ and only a few on $\{2\}$.
This is to be expected, as it requires more frequencies to approximate less smooth functions.
Furthermore, the same effect is observed within the ANOVA term $\{1,2\}$.

The  outcome for the experiment with noise is depicted in \Cref{s2_cv}.
\begin{figure} \centering
    \begingroup
  \makeatletter
  \providecommand\color[2][]{\GenericError{(gnuplot) \space\space\space\@spaces}{Package color not loaded in conjunction with
      terminal option `colourtext'}{See the gnuplot documentation for explanation.}{Either use 'blacktext' in gnuplot or load the package
      color.sty in LaTeX.}\renewcommand\color[2][]{}}\providecommand\includegraphics[2][]{\GenericError{(gnuplot) \space\space\space\@spaces}{Package graphicx or graphics not loaded}{See the gnuplot documentation for explanation.}{The gnuplot epslatex terminal needs graphicx.sty or graphics.sty.}\renewcommand\includegraphics[2][]{}}\providecommand\rotatebox[2]{#2}\@ifundefined{ifGPcolor}{\newif\ifGPcolor
    \GPcolortrue
  }{}\@ifundefined{ifGPblacktext}{\newif\ifGPblacktext
    \GPblacktexttrue
  }{}\let\gplgaddtomacro\g@addto@macro
\gdef\gplbacktext{}\gdef\gplfronttext{}\makeatother
  \ifGPblacktext
\def\colorrgb#1{}\def\colorgray#1{}\else
\ifGPcolor
      \def\colorrgb#1{\color[rgb]{#1}}\def\colorgray#1{\color[gray]{#1}}\expandafter\def\csname LTw\endcsname{\color{white}}\expandafter\def\csname LTb\endcsname{\color{black}}\expandafter\def\csname LTa\endcsname{\color{black}}\expandafter\def\csname LT0\endcsname{\color[rgb]{1,0,0}}\expandafter\def\csname LT1\endcsname{\color[rgb]{0,1,0}}\expandafter\def\csname LT2\endcsname{\color[rgb]{0,0,1}}\expandafter\def\csname LT3\endcsname{\color[rgb]{1,0,1}}\expandafter\def\csname LT4\endcsname{\color[rgb]{0,1,1}}\expandafter\def\csname LT5\endcsname{\color[rgb]{1,1,0}}\expandafter\def\csname LT6\endcsname{\color[rgb]{0,0,0}}\expandafter\def\csname LT7\endcsname{\color[rgb]{1,0.3,0}}\expandafter\def\csname LT8\endcsname{\color[rgb]{0.5,0.5,0.5}}\else
\def\colorrgb#1{\color{black}}\def\colorgray#1{\color[gray]{#1}}\expandafter\def\csname LTw\endcsname{\color{white}}\expandafter\def\csname LTb\endcsname{\color{black}}\expandafter\def\csname LTa\endcsname{\color{black}}\expandafter\def\csname LT0\endcsname{\color{black}}\expandafter\def\csname LT1\endcsname{\color{black}}\expandafter\def\csname LT2\endcsname{\color{black}}\expandafter\def\csname LT3\endcsname{\color{black}}\expandafter\def\csname LT4\endcsname{\color{black}}\expandafter\def\csname LT5\endcsname{\color{black}}\expandafter\def\csname LT6\endcsname{\color{black}}\expandafter\def\csname LT7\endcsname{\color{black}}\expandafter\def\csname LT8\endcsname{\color{black}}\fi
  \fi
    \setlength{\unitlength}{0.0500bp}\ifx\gptboxheight\undefined \newlength{\gptboxheight}\newlength{\gptboxwidth}\newsavebox{\gptboxtext}\fi \setlength{\fboxrule}{0.5pt}\setlength{\fboxsep}{1pt}\definecolor{tbcol}{rgb}{1,1,1}\begin{picture}(7920.00,2540.00)\gplgaddtomacro\gplbacktext{\colorrgb{0.00,0.00,0.00}\put(1711,599){\makebox(0,0)[r]{\strut{}\scriptsize $5\cdot 10^{-8}$}}\colorrgb{0.00,0.00,0.00}\put(1711,1041){\makebox(0,0)[r]{\strut{}\scriptsize $6\cdot 10^{-8}$}}\colorrgb{0.00,0.00,0.00}\put(1711,1415){\makebox(0,0)[r]{\strut{}\scriptsize $7\cdot 10^{-8}$}}\colorrgb{0.00,0.00,0.00}\put(1711,1739){\makebox(0,0)[r]{\strut{}\scriptsize $8\cdot 10^{-8}$}}\colorrgb{0.00,0.00,0.00}\put(1711,2024){\makebox(0,0)[r]{\strut{}\scriptsize $9\cdot 10^{-8}$}}\colorrgb{0.00,0.00,0.00}\put(1711,2280){\makebox(0,0)[r]{\strut{}\scriptsize $1\cdot 10^{-7}$}}\colorrgb{0.00,0.00,0.00}\put(2073,431){\makebox(0,0){\strut{}\scriptsize $10^{2}$}}\colorrgb{0.00,0.00,0.00}\put(2941,431){\makebox(0,0){\strut{}\scriptsize $10^{3}$}}\colorrgb{0.00,0.00,0.00}\put(3809,431){\makebox(0,0){\strut{}\scriptsize $10^{4}$}}\colorrgb{0.00,0.00,0.00}\put(4677,431){\makebox(0,0){\strut{}\scriptsize $10^{5}$}}}\gplgaddtomacro\gplfronttext{\csname LTb\endcsname \put(5382,2100){\makebox(0,0)[l]{\strut{}$\FCV(S_{\Psi_0(m)}^{\bm X}\bm y)$}}\csname LTb\endcsname \put(5382,1740){\makebox(0,0)[l]{\strut{}$\|f-S_{\Psi_0(m)}^{\bm X}\bm y\|_{L_2}^{2}+\sigma^2$}}\csname LTb\endcsname \put(5382,1381){\makebox(0,0)[l]{\strut{}$\FCV(S_{\Psi_1(m)}^{\bm X}\bm y)$}}\csname LTb\endcsname \put(5382,1021){\makebox(0,0)[l]{\strut{}$\|f-S_{\Psi_1(m)}^{\bm X}\bm y\|_{L_2}^{2}+\sigma^2$}}\csname LTb\endcsname \put(5382,662){\makebox(0,0)[l]{\strut{}$\FCV(S_{\Psi_2(m)}^{\bm X}\bm y)$}}\csname LTb\endcsname \put(5382,302){\makebox(0,0)[l]{\strut{}$\|f-S_{\Psi_2(m)}^{\bm X}\bm y\|_{L_2}^{2}+\sigma^2$}}\csname LTb\endcsname \put(3245,239){\makebox(0,0){\strut{}$m$}}}\gplbacktext
    \put(0,0){\includegraphics[width={396.00bp},height={127.00bp}]{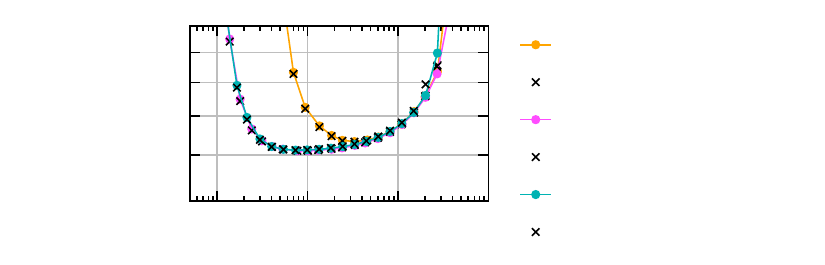}}\gplfronttext
  \end{picture}\endgroup
     \caption{Cross validation and $L_2$ error for the $d=2$ example with Gaussian noise.}\label{s2_cv}
\end{figure} Foremost, we observe that the fast cross-validation score $\FCV(S_{\mathcal I}^{\bm X}\bm y)$ is an excellent approximation for the $L_2$ error.
Further, we observe the expected under- and overfitting behavior.
With updating the frequency shape $\psi$, the overall error decay improves, which allows for a smaller error with fewer frequencies.
In the third iteration, we only observe a very slight improvement.
This aligns with our theoretical prediction from \Cref{fig:error_behavior}.

\subsection{Example with fixed superposition dimension} 

In the second example, we use the function
\begin{equation*}
    f(\bm x) = \frac{1}{a(\bm x)}
    \quad\text{with}\quad
    a(\bm x) = 1+\frac{1}{2}\sum_{j=1}^{d} j^{-q} \sin(2\pi x_j)
    \quad\text{and}\quad
    q=6
\end{equation*}
with spatial dimension $d=5$.
This function was considered in \cite{KKNS25,BGKS25} and solves the algebraic equation $a(\bm x)f(\bm x) = 1$, mimicking the features of a partial differential equation with a random coefficient whilst avoiding the complexity of a spatial variable or the need of a finite element solver.

For this function we restrict the ANOVA approximation to up to $3$-dimensional terms $U = \{\mathfrak u\in\mathcal P([d]) : |\mathfrak u| \le 3\}$.
For the noise-free experiment, the $L_2$ error is depicted in \Cref{s5_rates_l2error} and the frequency distribution in \Cref{s5_rates_box}.

\begin{figure} \centering
    \begingroup
  \makeatletter
  \providecommand\color[2][]{\GenericError{(gnuplot) \space\space\space\@spaces}{Package color not loaded in conjunction with
      terminal option `colourtext'}{See the gnuplot documentation for explanation.}{Either use 'blacktext' in gnuplot or load the package
      color.sty in LaTeX.}\renewcommand\color[2][]{}}\providecommand\includegraphics[2][]{\GenericError{(gnuplot) \space\space\space\@spaces}{Package graphicx or graphics not loaded}{See the gnuplot documentation for explanation.}{The gnuplot epslatex terminal needs graphicx.sty or graphics.sty.}\renewcommand\includegraphics[2][]{}}\providecommand\rotatebox[2]{#2}\@ifundefined{ifGPcolor}{\newif\ifGPcolor
    \GPcolortrue
  }{}\@ifundefined{ifGPblacktext}{\newif\ifGPblacktext
    \GPblacktexttrue
  }{}\let\gplgaddtomacro\g@addto@macro
\gdef\gplbacktext{}\gdef\gplfronttext{}\makeatother
  \ifGPblacktext
\def\colorrgb#1{}\def\colorgray#1{}\else
\ifGPcolor
      \def\colorrgb#1{\color[rgb]{#1}}\def\colorgray#1{\color[gray]{#1}}\expandafter\def\csname LTw\endcsname{\color{white}}\expandafter\def\csname LTb\endcsname{\color{black}}\expandafter\def\csname LTa\endcsname{\color{black}}\expandafter\def\csname LT0\endcsname{\color[rgb]{1,0,0}}\expandafter\def\csname LT1\endcsname{\color[rgb]{0,1,0}}\expandafter\def\csname LT2\endcsname{\color[rgb]{0,0,1}}\expandafter\def\csname LT3\endcsname{\color[rgb]{1,0,1}}\expandafter\def\csname LT4\endcsname{\color[rgb]{0,1,1}}\expandafter\def\csname LT5\endcsname{\color[rgb]{1,1,0}}\expandafter\def\csname LT6\endcsname{\color[rgb]{0,0,0}}\expandafter\def\csname LT7\endcsname{\color[rgb]{1,0.3,0}}\expandafter\def\csname LT8\endcsname{\color[rgb]{0.5,0.5,0.5}}\else
\def\colorrgb#1{\color{black}}\def\colorgray#1{\color[gray]{#1}}\expandafter\def\csname LTw\endcsname{\color{white}}\expandafter\def\csname LTb\endcsname{\color{black}}\expandafter\def\csname LTa\endcsname{\color{black}}\expandafter\def\csname LT0\endcsname{\color{black}}\expandafter\def\csname LT1\endcsname{\color{black}}\expandafter\def\csname LT2\endcsname{\color{black}}\expandafter\def\csname LT3\endcsname{\color{black}}\expandafter\def\csname LT4\endcsname{\color{black}}\expandafter\def\csname LT5\endcsname{\color{black}}\expandafter\def\csname LT6\endcsname{\color{black}}\expandafter\def\csname LT7\endcsname{\color{black}}\expandafter\def\csname LT8\endcsname{\color{black}}\fi
  \fi
    \setlength{\unitlength}{0.0500bp}\ifx\gptboxheight\undefined \newlength{\gptboxheight}\newlength{\gptboxwidth}\newsavebox{\gptboxtext}\fi \setlength{\fboxrule}{0.5pt}\setlength{\fboxsep}{1pt}\definecolor{tbcol}{rgb}{1,1,1}\begin{picture}(4520.00,2540.00)\gplgaddtomacro\gplbacktext{\colorrgb{0.00,0.00,0.00}\put(1107,575){\makebox(0,0)[r]{\strut{}\scriptsize $10^{-8}$}}\colorrgb{0.00,0.00,0.00}\put(1107,1427){\makebox(0,0)[r]{\strut{}\scriptsize $10^{-7}$}}\colorrgb{0.00,0.00,0.00}\put(1107,2280){\makebox(0,0)[r]{\strut{}\scriptsize $10^{-6}$}}\colorrgb{0.00,0.00,0.00}\put(1208,407){\makebox(0,0){\strut{}\scriptsize $1$}}\colorrgb{0.00,0.00,0.00}\put(1582,407){\makebox(0,0){\strut{}\scriptsize $2$}}\colorrgb{0.00,0.00,0.00}\put(1955,407){\makebox(0,0){\strut{}\scriptsize $3$}}\colorrgb{0.00,0.00,0.00}\put(2329,407){\makebox(0,0){\strut{}\scriptsize $4$}}\colorrgb{0.00,0.00,0.00}\put(2703,407){\makebox(0,0){\strut{}\scriptsize $5$}}\colorrgb{0.00,0.00,0.00}\put(3076,407){\makebox(0,0){\strut{}\scriptsize $6$}}\colorrgb{0.00,0.00,0.00}\put(3450,407){\makebox(0,0){\strut{}\scriptsize $7$}}\colorrgb{0.00,0.00,0.00}\put(3824,407){\makebox(0,0){\strut{}\scriptsize $8$}}\colorrgb{0.00,0.00,0.00}\put(4197,407){\makebox(0,0){\strut{}\scriptsize $9$}}}\gplgaddtomacro\gplfronttext{\csname LTb\endcsname \put(559,1427){\rotatebox{-270.00}{\makebox(0,0){\strut{}$L_2$ error}}}\csname LTb\endcsname \put(2703,167){\makebox(0,0){\strut{}iteration}}}\gplbacktext
    \put(0,0){\includegraphics[width={226.00bp},height={127.00bp}]{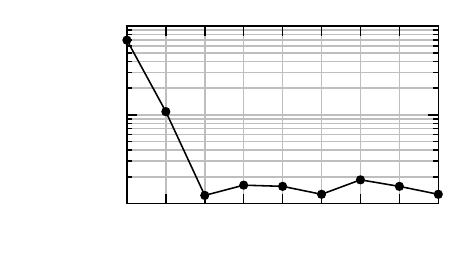}}\gplfronttext
  \end{picture}\endgroup
     \caption{$L_2$ error for the $d=5$ example.}\label{s5_rates_l2error}
\end{figure} 

\begin{figure} \centering
    \begingroup
  \makeatletter
  \providecommand\color[2][]{\GenericError{(gnuplot) \space\space\space\@spaces}{Package color not loaded in conjunction with
      terminal option `colourtext'}{See the gnuplot documentation for explanation.}{Either use 'blacktext' in gnuplot or load the package
      color.sty in LaTeX.}\renewcommand\color[2][]{}}\providecommand\includegraphics[2][]{\GenericError{(gnuplot) \space\space\space\@spaces}{Package graphicx or graphics not loaded}{See the gnuplot documentation for explanation.}{The gnuplot epslatex terminal needs graphicx.sty or graphics.sty.}\renewcommand\includegraphics[2][]{}}\providecommand\rotatebox[2]{#2}\@ifundefined{ifGPcolor}{\newif\ifGPcolor
    \GPcolortrue
  }{}\@ifundefined{ifGPblacktext}{\newif\ifGPblacktext
    \GPblacktexttrue
  }{}\let\gplgaddtomacro\g@addto@macro
\gdef\gplbacktext{}\gdef\gplfronttext{}\makeatother
  \ifGPblacktext
\def\colorrgb#1{}\def\colorgray#1{}\else
\ifGPcolor
      \def\colorrgb#1{\color[rgb]{#1}}\def\colorgray#1{\color[gray]{#1}}\expandafter\def\csname LTw\endcsname{\color{white}}\expandafter\def\csname LTb\endcsname{\color{black}}\expandafter\def\csname LTa\endcsname{\color{black}}\expandafter\def\csname LT0\endcsname{\color[rgb]{1,0,0}}\expandafter\def\csname LT1\endcsname{\color[rgb]{0,1,0}}\expandafter\def\csname LT2\endcsname{\color[rgb]{0,0,1}}\expandafter\def\csname LT3\endcsname{\color[rgb]{1,0,1}}\expandafter\def\csname LT4\endcsname{\color[rgb]{0,1,1}}\expandafter\def\csname LT5\endcsname{\color[rgb]{1,1,0}}\expandafter\def\csname LT6\endcsname{\color[rgb]{0,0,0}}\expandafter\def\csname LT7\endcsname{\color[rgb]{1,0.3,0}}\expandafter\def\csname LT8\endcsname{\color[rgb]{0.5,0.5,0.5}}\else
\def\colorrgb#1{\color{black}}\def\colorgray#1{\color[gray]{#1}}\expandafter\def\csname LTw\endcsname{\color{white}}\expandafter\def\csname LTb\endcsname{\color{black}}\expandafter\def\csname LTa\endcsname{\color{black}}\expandafter\def\csname LT0\endcsname{\color{black}}\expandafter\def\csname LT1\endcsname{\color{black}}\expandafter\def\csname LT2\endcsname{\color{black}}\expandafter\def\csname LT3\endcsname{\color{black}}\expandafter\def\csname LT4\endcsname{\color{black}}\expandafter\def\csname LT5\endcsname{\color{black}}\expandafter\def\csname LT6\endcsname{\color{black}}\expandafter\def\csname LT7\endcsname{\color{black}}\expandafter\def\csname LT8\endcsname{\color{black}}\fi
  \fi
    \setlength{\unitlength}{0.0500bp}\ifx\gptboxheight\undefined \newlength{\gptboxheight}\newlength{\gptboxwidth}\newsavebox{\gptboxtext}\fi \setlength{\fboxrule}{0.5pt}\setlength{\fboxsep}{1pt}\definecolor{tbcol}{rgb}{1,1,1}\begin{picture}(7920.00,2260.00)\gplgaddtomacro\gplbacktext{}\gplgaddtomacro\gplfronttext{\csname LTb\endcsname \put(1488,313){\makebox(0,0){\strut{}\scriptsize 1}}\csname LTb\endcsname \put(1488,940){\makebox(0,0){\strut{}\scriptsize 2}}\csname LTb\endcsname \put(1488,1566){\makebox(0,0){\strut{}\scriptsize 3}}\csname LTb\endcsname \put(1697,313){\makebox(0,0){\strut{}\scriptsize 1}}\csname LTb\endcsname \put(1697,940){\makebox(0,0){\strut{}\scriptsize 2}}\csname LTb\endcsname \put(1697,1566){\makebox(0,0){\strut{}\scriptsize 4}}\csname LTb\endcsname \put(1907,313){\makebox(0,0){\strut{}\scriptsize 1}}\csname LTb\endcsname \put(1907,940){\makebox(0,0){\strut{}\scriptsize 2}}\csname LTb\endcsname \put(1907,1566){\makebox(0,0){\strut{}\scriptsize 5}}\csname LTb\endcsname \put(2117,313){\makebox(0,0){\strut{}\scriptsize 1}}\csname LTb\endcsname \put(2117,940){\makebox(0,0){\strut{}\scriptsize 3}}\csname LTb\endcsname \put(2117,1566){\makebox(0,0){\strut{}\scriptsize 4}}\csname LTb\endcsname \put(2326,313){\makebox(0,0){\strut{}\scriptsize 1}}\csname LTb\endcsname \put(2326,940){\makebox(0,0){\strut{}\scriptsize 3}}\csname LTb\endcsname \put(2326,1566){\makebox(0,0){\strut{}\scriptsize 5}}\csname LTb\endcsname \put(2536,313){\makebox(0,0){\strut{}\scriptsize 1}}\csname LTb\endcsname \put(2536,940){\makebox(0,0){\strut{}\scriptsize 4}}\csname LTb\endcsname \put(2536,1566){\makebox(0,0){\strut{}\scriptsize 5}}\csname LTb\endcsname \put(2745,313){\makebox(0,0){\strut{}\scriptsize 2}}\csname LTb\endcsname \put(2745,940){\makebox(0,0){\strut{}\scriptsize 3}}\csname LTb\endcsname \put(2745,1566){\makebox(0,0){\strut{}\scriptsize 4}}\csname LTb\endcsname \put(2955,313){\makebox(0,0){\strut{}\scriptsize 2}}\csname LTb\endcsname \put(2955,940){\makebox(0,0){\strut{}\scriptsize 3}}\csname LTb\endcsname \put(2955,1566){\makebox(0,0){\strut{}\scriptsize 5}}\csname LTb\endcsname \put(3164,313){\makebox(0,0){\strut{}\scriptsize 2}}\csname LTb\endcsname \put(3164,940){\makebox(0,0){\strut{}\scriptsize 4}}\csname LTb\endcsname \put(3164,1566){\makebox(0,0){\strut{}\scriptsize 5}}\csname LTb\endcsname \put(3374,313){\makebox(0,0){\strut{}\scriptsize 3}}\csname LTb\endcsname \put(3374,940){\makebox(0,0){\strut{}\scriptsize 4}}\csname LTb\endcsname \put(3374,1566){\makebox(0,0){\strut{}\scriptsize 5}}\csname LTb\endcsname \put(1974,2096){\makebox(0,0){\strut{}initial frequency distribution}}}\gplgaddtomacro\gplbacktext{}\gplgaddtomacro\gplfronttext{\csname LTb\endcsname \put(4544,876){\makebox(0,0){\strut{}\scriptsize 1}}\csname LTb\endcsname \put(4544,1817){\makebox(0,0){\strut{}\scriptsize 2}}\csname LTb\endcsname \put(5170,824){\makebox(0,0){\strut{}\scriptsize 1}}\csname LTb\endcsname \put(5170,1706){\makebox(0,0){\strut{}\scriptsize 2}}\csname LTb\endcsname \put(5170,1822){\makebox(0,0){\strut{}\scriptsize 3}}\csname LTb\endcsname \put(5645,821){\makebox(0,0){\strut{}\scriptsize 1}}\csname LTb\endcsname \put(5645,1701){\makebox(0,0){\strut{}\scriptsize 2}}\csname LTb\endcsname \put(5645,1820){\makebox(0,0){\strut{}\scriptsize 4}}\csname LTb\endcsname \put(6087,807){\makebox(0,0){\strut{}\scriptsize 1}}\csname LTb\endcsname \put(6087,1681){\makebox(0,0){\strut{}\scriptsize 2}}\csname LTb\endcsname \put(6087,1814){\makebox(0,0){\strut{}\scriptsize 5}}\csname LTb\endcsname \put(6488,799){\makebox(0,0){\strut{}\scriptsize 1}}\csname LTb\endcsname \put(6488,1669){\makebox(0,0){\strut{}\scriptsize 3}}\csname LTb\endcsname \put(6488,1810){\makebox(0,0){\strut{}\scriptsize 4}}\csname LTb\endcsname \put(6841,775){\makebox(0,0){\strut{}\scriptsize 1}}\csname LTb\endcsname \put(6841,1632){\makebox(0,0){\strut{}\scriptsize 3}}\csname LTb\endcsname \put(6841,1797){\makebox(0,0){\strut{}\scriptsize 5}}\csname LTb\endcsname \put(7147,762){\makebox(0,0){\strut{}\scriptsize 1}}\csname LTb\endcsname \put(7147,1614){\makebox(0,0){\strut{}\scriptsize 4}}\csname LTb\endcsname \put(7147,1791){\makebox(0,0){\strut{}\scriptsize 5}}\csname LTb\endcsname \put(5924,2096){\makebox(0,0){\strut{}final frequency distribution}}}\gplbacktext
    \put(0,0){\includegraphics[width={396.00bp},height={113.00bp}]{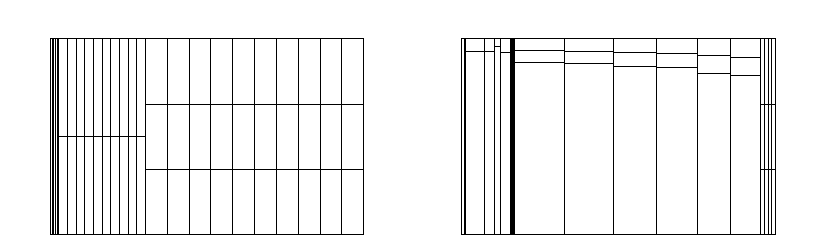}}\gplfronttext
  \end{picture}\endgroup
     \caption{Depiction of the frequency distribution in iteration $1$ and $9$ in the ANOVA terms for the $d=5$ example.}\label{s5_rates_box}
\end{figure} 

In the first and second iterations, the $L_2$ error lessens by a factor of $10$ each before it stabilizes.
This shows the effect of a better approximation yielding a better estimation of the smoothness parameters, which in turn yields a better approximation before a fixed point is reached.
In the frequency distribution we see that more frequencies are spent for smaller dimensions.
This is to be expected, as the function has decaying weights with increasing dimension.

When noise is added, we obtain the outcome depicted in \Cref{s2_cv}.
\begin{figure} \centering
    \begingroup
  \makeatletter
  \providecommand\color[2][]{\GenericError{(gnuplot) \space\space\space\@spaces}{Package color not loaded in conjunction with
      terminal option `colourtext'}{See the gnuplot documentation for explanation.}{Either use 'blacktext' in gnuplot or load the package
      color.sty in LaTeX.}\renewcommand\color[2][]{}}\providecommand\includegraphics[2][]{\GenericError{(gnuplot) \space\space\space\@spaces}{Package graphicx or graphics not loaded}{See the gnuplot documentation for explanation.}{The gnuplot epslatex terminal needs graphicx.sty or graphics.sty.}\renewcommand\includegraphics[2][]{}}\providecommand\rotatebox[2]{#2}\@ifundefined{ifGPcolor}{\newif\ifGPcolor
    \GPcolortrue
  }{}\@ifundefined{ifGPblacktext}{\newif\ifGPblacktext
    \GPblacktexttrue
  }{}\let\gplgaddtomacro\g@addto@macro
\gdef\gplbacktext{}\gdef\gplfronttext{}\makeatother
  \ifGPblacktext
\def\colorrgb#1{}\def\colorgray#1{}\else
\ifGPcolor
      \def\colorrgb#1{\color[rgb]{#1}}\def\colorgray#1{\color[gray]{#1}}\expandafter\def\csname LTw\endcsname{\color{white}}\expandafter\def\csname LTb\endcsname{\color{black}}\expandafter\def\csname LTa\endcsname{\color{black}}\expandafter\def\csname LT0\endcsname{\color[rgb]{1,0,0}}\expandafter\def\csname LT1\endcsname{\color[rgb]{0,1,0}}\expandafter\def\csname LT2\endcsname{\color[rgb]{0,0,1}}\expandafter\def\csname LT3\endcsname{\color[rgb]{1,0,1}}\expandafter\def\csname LT4\endcsname{\color[rgb]{0,1,1}}\expandafter\def\csname LT5\endcsname{\color[rgb]{1,1,0}}\expandafter\def\csname LT6\endcsname{\color[rgb]{0,0,0}}\expandafter\def\csname LT7\endcsname{\color[rgb]{1,0.3,0}}\expandafter\def\csname LT8\endcsname{\color[rgb]{0.5,0.5,0.5}}\else
\def\colorrgb#1{\color{black}}\def\colorgray#1{\color[gray]{#1}}\expandafter\def\csname LTw\endcsname{\color{white}}\expandafter\def\csname LTb\endcsname{\color{black}}\expandafter\def\csname LTa\endcsname{\color{black}}\expandafter\def\csname LT0\endcsname{\color{black}}\expandafter\def\csname LT1\endcsname{\color{black}}\expandafter\def\csname LT2\endcsname{\color{black}}\expandafter\def\csname LT3\endcsname{\color{black}}\expandafter\def\csname LT4\endcsname{\color{black}}\expandafter\def\csname LT5\endcsname{\color{black}}\expandafter\def\csname LT6\endcsname{\color{black}}\expandafter\def\csname LT7\endcsname{\color{black}}\expandafter\def\csname LT8\endcsname{\color{black}}\fi
  \fi
    \setlength{\unitlength}{0.0500bp}\ifx\gptboxheight\undefined \newlength{\gptboxheight}\newlength{\gptboxwidth}\newsavebox{\gptboxtext}\fi \setlength{\fboxrule}{0.5pt}\setlength{\fboxsep}{1pt}\definecolor{tbcol}{rgb}{1,1,1}\begin{picture}(7920.00,2540.00)\gplgaddtomacro\gplbacktext{\colorrgb{0.00,0.00,0.00}\put(1711,599){\makebox(0,0)[r]{\strut{}\scriptsize $1\cdot 10^{-5}$}}\colorrgb{0.00,0.00,0.00}\put(1711,1439){\makebox(0,0)[r]{\strut{}\scriptsize $2\cdot 10^{-5}$}}\colorrgb{0.00,0.00,0.00}\put(1711,1931){\makebox(0,0)[r]{\strut{}\scriptsize $3\cdot 10^{-5}$}}\colorrgb{0.00,0.00,0.00}\put(1711,2280){\makebox(0,0)[r]{\strut{}\scriptsize $4\cdot 10^{-5}$}}\colorrgb{0.00,0.00,0.00}\put(2406,431){\makebox(0,0){\strut{}\scriptsize $10^{3}$}}\colorrgb{0.00,0.00,0.00}\put(3541,431){\makebox(0,0){\strut{}\scriptsize $10^{4}$}}\colorrgb{0.00,0.00,0.00}\put(4677,431){\makebox(0,0){\strut{}\scriptsize $10^{5}$}}}\gplgaddtomacro\gplfronttext{\csname LTb\endcsname \put(5382,2100){\makebox(0,0)[l]{\strut{}$\FCV(S_{\Psi_0(m)}^{\bm X}\bm y)$}}\csname LTb\endcsname \put(5382,1740){\makebox(0,0)[l]{\strut{}$\|f-S_{\Psi_0(m)}^{\bm X}\bm y\|_{L_2}^{2}+\sigma^2$}}\csname LTb\endcsname \put(5382,1381){\makebox(0,0)[l]{\strut{}$\FCV(S_{\Psi_1(m)}^{\bm X}\bm y)$}}\csname LTb\endcsname \put(5382,1021){\makebox(0,0)[l]{\strut{}$\|f-S_{\Psi_1(m)}^{\bm X}\bm y\|_{L_2}^{2}+\sigma^2$}}\csname LTb\endcsname \put(5382,662){\makebox(0,0)[l]{\strut{}$\FCV(S_{\Psi_2(m)}^{\bm X}\bm y)$}}\csname LTb\endcsname \put(5382,302){\makebox(0,0)[l]{\strut{}$\|f-S_{\Psi_2(m)}^{\bm X}\bm y\|_{L_2}^{2}+\sigma^2$}}\csname LTb\endcsname \put(3245,239){\makebox(0,0){\strut{}$m$}}}\gplbacktext
    \put(0,0){\includegraphics[width={396.00bp},height={127.00bp}]{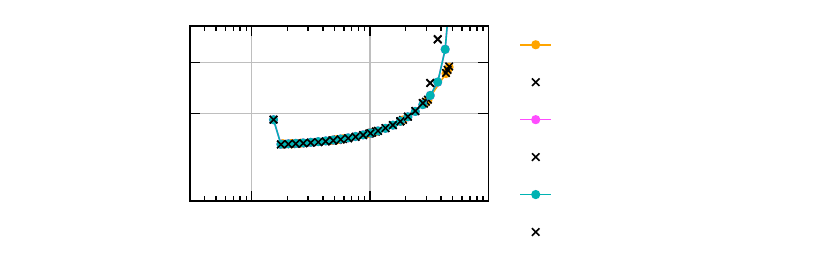}}\gplfronttext
  \end{picture}\endgroup
     \caption{Cross validation and $L_2$ error for the $d=5$ example with Gaussian noise.}\label{s5_cv}
\end{figure} Here we are restricted by a high minimum number of frequencies, as the number of ANOVA terms is high and we use at least $5^{|\mathfrak u|}$ frequencies in each to make decay rates detectable.
With this high number of frequencies we are forced to only work in the overfitting regime, where the number of points and frequencies dominates the error behavior but not the shape, which is why we do not see an improvement.
This could be improved by manually omitting small ANOVA terms in terms of the $L_2$ norm or global sensitivity indices, cf.~\Cref{sec:anova}.

\subsection{Example with known ANOVA terms} 

The third example is a $9$-dimensional combination of B-splines
\begin{align*}
    f(\bm x) &= \frac{1}{4.617\dots} \Big( B_2(x_1)B_4(x_2)B_6(x_3) + B_2(x_4)B_4(x_5) + B_6(x_5)B_2(x_6) \\
    &\phantom=+ B_4(x_6)B_6(x_7) + B_2(x_7)B_4(x_8) + B_6(x_8)B_2(x_9) + B_4(x_9)B_6(x_{10}) \Big) \,.
\end{align*}
Functions of this type were already used in \cite{PS21,BPS22,PS22}.
The B-spline of order $n$ is a piecewise polynomial of order $n$, which has smoothness $s = n-1/2$.
In this example we assume to know the existing ANOVA terms, which could be determined via global sensitivity indices, cf.~\Cref{sec:anova}.
In contrast to \eqref{eq:2}, for a product of B-splines the lower-dimensional ANOVA terms have to be included as well.

For the noise-free experiment, the $L_2$ error is depicted in \Cref{s10_rates_l2error} and the frequency distribution in \Cref{s10_rates_box}.

\begin{figure} \centering
    \begingroup
  \makeatletter
  \providecommand\color[2][]{\GenericError{(gnuplot) \space\space\space\@spaces}{Package color not loaded in conjunction with
      terminal option `colourtext'}{See the gnuplot documentation for explanation.}{Either use 'blacktext' in gnuplot or load the package
      color.sty in LaTeX.}\renewcommand\color[2][]{}}\providecommand\includegraphics[2][]{\GenericError{(gnuplot) \space\space\space\@spaces}{Package graphicx or graphics not loaded}{See the gnuplot documentation for explanation.}{The gnuplot epslatex terminal needs graphicx.sty or graphics.sty.}\renewcommand\includegraphics[2][]{}}\providecommand\rotatebox[2]{#2}\@ifundefined{ifGPcolor}{\newif\ifGPcolor
    \GPcolortrue
  }{}\@ifundefined{ifGPblacktext}{\newif\ifGPblacktext
    \GPblacktexttrue
  }{}\let\gplgaddtomacro\g@addto@macro
\gdef\gplbacktext{}\gdef\gplfronttext{}\makeatother
  \ifGPblacktext
\def\colorrgb#1{}\def\colorgray#1{}\else
\ifGPcolor
      \def\colorrgb#1{\color[rgb]{#1}}\def\colorgray#1{\color[gray]{#1}}\expandafter\def\csname LTw\endcsname{\color{white}}\expandafter\def\csname LTb\endcsname{\color{black}}\expandafter\def\csname LTa\endcsname{\color{black}}\expandafter\def\csname LT0\endcsname{\color[rgb]{1,0,0}}\expandafter\def\csname LT1\endcsname{\color[rgb]{0,1,0}}\expandafter\def\csname LT2\endcsname{\color[rgb]{0,0,1}}\expandafter\def\csname LT3\endcsname{\color[rgb]{1,0,1}}\expandafter\def\csname LT4\endcsname{\color[rgb]{0,1,1}}\expandafter\def\csname LT5\endcsname{\color[rgb]{1,1,0}}\expandafter\def\csname LT6\endcsname{\color[rgb]{0,0,0}}\expandafter\def\csname LT7\endcsname{\color[rgb]{1,0.3,0}}\expandafter\def\csname LT8\endcsname{\color[rgb]{0.5,0.5,0.5}}\else
\def\colorrgb#1{\color{black}}\def\colorgray#1{\color[gray]{#1}}\expandafter\def\csname LTw\endcsname{\color{white}}\expandafter\def\csname LTb\endcsname{\color{black}}\expandafter\def\csname LTa\endcsname{\color{black}}\expandafter\def\csname LT0\endcsname{\color{black}}\expandafter\def\csname LT1\endcsname{\color{black}}\expandafter\def\csname LT2\endcsname{\color{black}}\expandafter\def\csname LT3\endcsname{\color{black}}\expandafter\def\csname LT4\endcsname{\color{black}}\expandafter\def\csname LT5\endcsname{\color{black}}\expandafter\def\csname LT6\endcsname{\color{black}}\expandafter\def\csname LT7\endcsname{\color{black}}\expandafter\def\csname LT8\endcsname{\color{black}}\fi
  \fi
    \setlength{\unitlength}{0.0500bp}\ifx\gptboxheight\undefined \newlength{\gptboxheight}\newlength{\gptboxwidth}\newsavebox{\gptboxtext}\fi \setlength{\fboxrule}{0.5pt}\setlength{\fboxsep}{1pt}\definecolor{tbcol}{rgb}{1,1,1}\begin{picture}(4520.00,2540.00)\gplgaddtomacro\gplbacktext{\colorrgb{0.00,0.00,0.00}\put(1107,575){\makebox(0,0)[r]{\strut{}\scriptsize $10^{-4}$}}\colorrgb{0.00,0.00,0.00}\put(1107,1427){\makebox(0,0)[r]{\strut{}\scriptsize $10^{-3}$}}\colorrgb{0.00,0.00,0.00}\put(1107,2280){\makebox(0,0)[r]{\strut{}\scriptsize $10^{-2}$}}\colorrgb{0.00,0.00,0.00}\put(1208,407){\makebox(0,0){\strut{}\scriptsize $1$}}\colorrgb{0.00,0.00,0.00}\put(1582,407){\makebox(0,0){\strut{}\scriptsize $2$}}\colorrgb{0.00,0.00,0.00}\put(1955,407){\makebox(0,0){\strut{}\scriptsize $3$}}\colorrgb{0.00,0.00,0.00}\put(2329,407){\makebox(0,0){\strut{}\scriptsize $4$}}\colorrgb{0.00,0.00,0.00}\put(2703,407){\makebox(0,0){\strut{}\scriptsize $5$}}\colorrgb{0.00,0.00,0.00}\put(3076,407){\makebox(0,0){\strut{}\scriptsize $6$}}\colorrgb{0.00,0.00,0.00}\put(3450,407){\makebox(0,0){\strut{}\scriptsize $7$}}\colorrgb{0.00,0.00,0.00}\put(3824,407){\makebox(0,0){\strut{}\scriptsize $8$}}\colorrgb{0.00,0.00,0.00}\put(4197,407){\makebox(0,0){\strut{}\scriptsize $9$}}}\gplgaddtomacro\gplfronttext{\csname LTb\endcsname \put(559,1427){\rotatebox{-270.00}{\makebox(0,0){\strut{}$L_2$ error}}}\csname LTb\endcsname \put(2703,167){\makebox(0,0){\strut{}iteration}}}\gplbacktext
    \put(0,0){\includegraphics[width={226.00bp},height={127.00bp}]{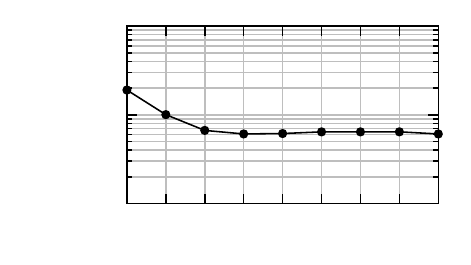}}\gplfronttext
  \end{picture}\endgroup
     \caption{$L_2$ error for the $d=10$ example.}\label{s10_rates_l2error}
\end{figure} 

\begin{figure} \centering
    \begingroup
  \makeatletter
  \providecommand\color[2][]{\GenericError{(gnuplot) \space\space\space\@spaces}{Package color not loaded in conjunction with
      terminal option `colourtext'}{See the gnuplot documentation for explanation.}{Either use 'blacktext' in gnuplot or load the package
      color.sty in LaTeX.}\renewcommand\color[2][]{}}\providecommand\includegraphics[2][]{\GenericError{(gnuplot) \space\space\space\@spaces}{Package graphicx or graphics not loaded}{See the gnuplot documentation for explanation.}{The gnuplot epslatex terminal needs graphicx.sty or graphics.sty.}\renewcommand\includegraphics[2][]{}}\providecommand\rotatebox[2]{#2}\@ifundefined{ifGPcolor}{\newif\ifGPcolor
    \GPcolortrue
  }{}\@ifundefined{ifGPblacktext}{\newif\ifGPblacktext
    \GPblacktexttrue
  }{}\let\gplgaddtomacro\g@addto@macro
\gdef\gplbacktext{}\gdef\gplfronttext{}\makeatother
  \ifGPblacktext
\def\colorrgb#1{}\def\colorgray#1{}\else
\ifGPcolor
      \def\colorrgb#1{\color[rgb]{#1}}\def\colorgray#1{\color[gray]{#1}}\expandafter\def\csname LTw\endcsname{\color{white}}\expandafter\def\csname LTb\endcsname{\color{black}}\expandafter\def\csname LTa\endcsname{\color{black}}\expandafter\def\csname LT0\endcsname{\color[rgb]{1,0,0}}\expandafter\def\csname LT1\endcsname{\color[rgb]{0,1,0}}\expandafter\def\csname LT2\endcsname{\color[rgb]{0,0,1}}\expandafter\def\csname LT3\endcsname{\color[rgb]{1,0,1}}\expandafter\def\csname LT4\endcsname{\color[rgb]{0,1,1}}\expandafter\def\csname LT5\endcsname{\color[rgb]{1,1,0}}\expandafter\def\csname LT6\endcsname{\color[rgb]{0,0,0}}\expandafter\def\csname LT7\endcsname{\color[rgb]{1,0.3,0}}\expandafter\def\csname LT8\endcsname{\color[rgb]{0.5,0.5,0.5}}\else
\def\colorrgb#1{\color{black}}\def\colorgray#1{\color[gray]{#1}}\expandafter\def\csname LTw\endcsname{\color{white}}\expandafter\def\csname LTb\endcsname{\color{black}}\expandafter\def\csname LTa\endcsname{\color{black}}\expandafter\def\csname LT0\endcsname{\color{black}}\expandafter\def\csname LT1\endcsname{\color{black}}\expandafter\def\csname LT2\endcsname{\color{black}}\expandafter\def\csname LT3\endcsname{\color{black}}\expandafter\def\csname LT4\endcsname{\color{black}}\expandafter\def\csname LT5\endcsname{\color{black}}\expandafter\def\csname LT6\endcsname{\color{black}}\expandafter\def\csname LT7\endcsname{\color{black}}\expandafter\def\csname LT8\endcsname{\color{black}}\fi
  \fi
    \setlength{\unitlength}{0.0500bp}\ifx\gptboxheight\undefined \newlength{\gptboxheight}\newlength{\gptboxwidth}\newsavebox{\gptboxtext}\fi \setlength{\fboxrule}{0.5pt}\setlength{\fboxsep}{1pt}\definecolor{tbcol}{rgb}{1,1,1}\begin{picture}(7920.00,2260.00)\gplgaddtomacro\gplbacktext{}\gplgaddtomacro\gplfronttext{\csname LTb\endcsname \put(889,470){\makebox(0,0){\strut{}\scriptsize 1}}\csname LTb\endcsname \put(889,1410){\makebox(0,0){\strut{}\scriptsize 2}}\csname LTb\endcsname \put(1122,470){\makebox(0,0){\strut{}\scriptsize 1}}\csname LTb\endcsname \put(1122,1410){\makebox(0,0){\strut{}\scriptsize 3}}\csname LTb\endcsname \put(1355,470){\makebox(0,0){\strut{}\scriptsize 2}}\csname LTb\endcsname \put(1355,1410){\makebox(0,0){\strut{}\scriptsize 3}}\csname LTb\endcsname \put(1588,470){\makebox(0,0){\strut{}\scriptsize 4}}\csname LTb\endcsname \put(1588,1410){\makebox(0,0){\strut{}\scriptsize 5}}\csname LTb\endcsname \put(1821,470){\makebox(0,0){\strut{}\scriptsize 5}}\csname LTb\endcsname \put(1821,1410){\makebox(0,0){\strut{}\scriptsize 6}}\csname LTb\endcsname \put(2054,470){\makebox(0,0){\strut{}\scriptsize 6}}\csname LTb\endcsname \put(2054,1410){\makebox(0,0){\strut{}\scriptsize 7}}\csname LTb\endcsname \put(2287,470){\makebox(0,0){\strut{}\scriptsize 7}}\csname LTb\endcsname \put(2287,1410){\makebox(0,0){\strut{}\scriptsize 8}}\csname LTb\endcsname \put(2520,470){\makebox(0,0){\strut{}\scriptsize 8}}\csname LTb\endcsname \put(2520,1410){\makebox(0,0){\strut{}\scriptsize 9}}\csname LTb\endcsname \put(2753,470){\makebox(0,0){\strut{}\scriptsize 9}}\csname LTb\endcsname \put(2753,1410){\makebox(0,0){\strut{}\scriptsize 10}}\csname LTb\endcsname \put(3174,313){\makebox(0,0){\strut{}\scriptsize 1}}\csname LTb\endcsname \put(3174,940){\makebox(0,0){\strut{}\scriptsize 2}}\csname LTb\endcsname \put(3174,1566){\makebox(0,0){\strut{}\scriptsize 3}}\csname LTb\endcsname \put(1974,2096){\makebox(0,0){\strut{}initial frequency distribution}}}\gplgaddtomacro\gplbacktext{}\gplgaddtomacro\gplfronttext{\csname LTb\endcsname \put(4770,868){\makebox(0,0){\strut{}\scriptsize 1}}\csname LTb\endcsname \put(4770,1808){\makebox(0,0){\strut{}\scriptsize 2}}\csname LTb\endcsname \put(4984,832){\makebox(0,0){\strut{}\scriptsize 1}}\csname LTb\endcsname \put(4984,1772){\makebox(0,0){\strut{}\scriptsize 3}}\csname LTb\endcsname \put(5279,884){\makebox(0,0){\strut{}\scriptsize 4}}\csname LTb\endcsname \put(5279,1824){\makebox(0,0){\strut{}\scriptsize 5}}\csname LTb\endcsname \put(5541,94){\makebox(0,0){\strut{}\scriptsize 5}}\csname LTb\endcsname \put(5541,1034){\makebox(0,0){\strut{}\scriptsize 6}}\csname LTb\endcsname \put(5857,884){\makebox(0,0){\strut{}\scriptsize 7}}\csname LTb\endcsname \put(5857,1824){\makebox(0,0){\strut{}\scriptsize 8}}\csname LTb\endcsname \put(6125,90){\makebox(0,0){\strut{}\scriptsize 8}}\csname LTb\endcsname \put(6125,1030){\makebox(0,0){\strut{}\scriptsize 9}}\csname LTb\endcsname \put(6884,589){\makebox(0,0){\strut{}\scriptsize 1}}\csname LTb\endcsname \put(6884,1290){\makebox(0,0){\strut{}\scriptsize 2}}\csname LTb\endcsname \put(6884,1641){\makebox(0,0){\strut{}\scriptsize 3}}\csname LTb\endcsname \put(5924,2096){\makebox(0,0){\strut{}final frequency distribution}}}\gplbacktext
    \put(0,0){\includegraphics[width={396.00bp},height={113.00bp}]{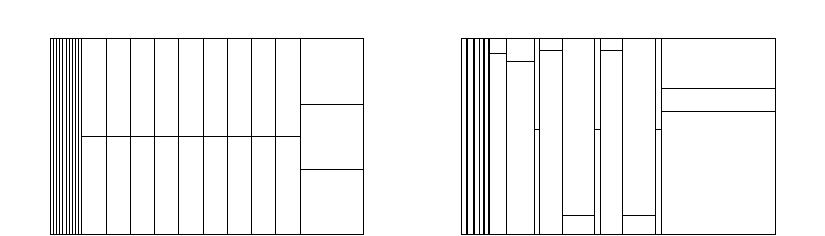}}\gplfronttext
  \end{picture}\endgroup
     \caption{Depiction of the frequency distribution in iteration $1$ and $9$ in the ANOVA terms for the $d=10$ example.}\label{s10_rates_box}
\end{figure} 

We observe an improvement of the $L_2$ error with the first $2$ iterations before it stabilizes.
The overall improvement is not as good as in the previous experiments, with an overall factor of $3$ in the $L_2$ norm accuracy.
The shape of the final frequencies again resembles the respective smoothness of the approximated function.

The outcome for the experiment with noise is depicted in \Cref{s10_cv}.
\begin{figure} \centering
    \begingroup
  \makeatletter
  \providecommand\color[2][]{\GenericError{(gnuplot) \space\space\space\@spaces}{Package color not loaded in conjunction with
      terminal option `colourtext'}{See the gnuplot documentation for explanation.}{Either use 'blacktext' in gnuplot or load the package
      color.sty in LaTeX.}\renewcommand\color[2][]{}}\providecommand\includegraphics[2][]{\GenericError{(gnuplot) \space\space\space\@spaces}{Package graphicx or graphics not loaded}{See the gnuplot documentation for explanation.}{The gnuplot epslatex terminal needs graphicx.sty or graphics.sty.}\renewcommand\includegraphics[2][]{}}\providecommand\rotatebox[2]{#2}\@ifundefined{ifGPcolor}{\newif\ifGPcolor
    \GPcolortrue
  }{}\@ifundefined{ifGPblacktext}{\newif\ifGPblacktext
    \GPblacktexttrue
  }{}\let\gplgaddtomacro\g@addto@macro
\gdef\gplbacktext{}\gdef\gplfronttext{}\makeatother
  \ifGPblacktext
\def\colorrgb#1{}\def\colorgray#1{}\else
\ifGPcolor
      \def\colorrgb#1{\color[rgb]{#1}}\def\colorgray#1{\color[gray]{#1}}\expandafter\def\csname LTw\endcsname{\color{white}}\expandafter\def\csname LTb\endcsname{\color{black}}\expandafter\def\csname LTa\endcsname{\color{black}}\expandafter\def\csname LT0\endcsname{\color[rgb]{1,0,0}}\expandafter\def\csname LT1\endcsname{\color[rgb]{0,1,0}}\expandafter\def\csname LT2\endcsname{\color[rgb]{0,0,1}}\expandafter\def\csname LT3\endcsname{\color[rgb]{1,0,1}}\expandafter\def\csname LT4\endcsname{\color[rgb]{0,1,1}}\expandafter\def\csname LT5\endcsname{\color[rgb]{1,1,0}}\expandafter\def\csname LT6\endcsname{\color[rgb]{0,0,0}}\expandafter\def\csname LT7\endcsname{\color[rgb]{1,0.3,0}}\expandafter\def\csname LT8\endcsname{\color[rgb]{0.5,0.5,0.5}}\else
\def\colorrgb#1{\color{black}}\def\colorgray#1{\color[gray]{#1}}\expandafter\def\csname LTw\endcsname{\color{white}}\expandafter\def\csname LTb\endcsname{\color{black}}\expandafter\def\csname LTa\endcsname{\color{black}}\expandafter\def\csname LT0\endcsname{\color{black}}\expandafter\def\csname LT1\endcsname{\color{black}}\expandafter\def\csname LT2\endcsname{\color{black}}\expandafter\def\csname LT3\endcsname{\color{black}}\expandafter\def\csname LT4\endcsname{\color{black}}\expandafter\def\csname LT5\endcsname{\color{black}}\expandafter\def\csname LT6\endcsname{\color{black}}\expandafter\def\csname LT7\endcsname{\color{black}}\expandafter\def\csname LT8\endcsname{\color{black}}\fi
  \fi
    \setlength{\unitlength}{0.0500bp}\ifx\gptboxheight\undefined \newlength{\gptboxheight}\newlength{\gptboxwidth}\newsavebox{\gptboxtext}\fi \setlength{\fboxrule}{0.5pt}\setlength{\fboxsep}{1pt}\definecolor{tbcol}{rgb}{1,1,1}\begin{picture}(7920.00,2540.00)\gplgaddtomacro\gplbacktext{\colorrgb{0.00,0.00,0.00}\put(1711,599){\makebox(0,0)[r]{\strut{}\scriptsize $2\cdot 10^{-4}$}}\colorrgb{0.00,0.00,0.00}\put(1711,1022){\makebox(0,0)[r]{\strut{}\scriptsize $3\cdot 10^{-4}$}}\colorrgb{0.00,0.00,0.00}\put(1711,1323){\makebox(0,0)[r]{\strut{}\scriptsize $4\cdot 10^{-4}$}}\colorrgb{0.00,0.00,0.00}\put(1711,1556){\makebox(0,0)[r]{\strut{}\scriptsize $5\cdot 10^{-4}$}}\colorrgb{0.00,0.00,0.00}\put(1711,1746){\makebox(0,0)[r]{\strut{}\scriptsize $6\cdot 10^{-4}$}}\colorrgb{0.00,0.00,0.00}\put(1711,1907){\makebox(0,0)[r]{\strut{}\scriptsize $7\cdot 10^{-4}$}}\colorrgb{0.00,0.00,0.00}\put(1711,2047){\makebox(0,0)[r]{\strut{}\scriptsize $8\cdot 10^{-4}$}}\colorrgb{0.00,0.00,0.00}\put(1711,2170){\makebox(0,0)[r]{\strut{}\scriptsize $9\cdot 10^{-4}$}}\colorrgb{0.00,0.00,0.00}\put(1711,2280){\makebox(0,0)[r]{\strut{}\scriptsize $1\cdot 10^{-3}$}}\colorrgb{0.00,0.00,0.00}\put(1812,431){\makebox(0,0){\strut{}\scriptsize $10^{3}$}}\colorrgb{0.00,0.00,0.00}\put(3245,431){\makebox(0,0){\strut{}\scriptsize $10^{4}$}}\colorrgb{0.00,0.00,0.00}\put(4677,431){\makebox(0,0){\strut{}\scriptsize $10^{5}$}}}\gplgaddtomacro\gplfronttext{\csname LTb\endcsname \put(5382,2100){\makebox(0,0)[l]{\strut{}$\FCV(S_{\Psi_0(m)}^{\bm X}\bm y)$}}\csname LTb\endcsname \put(5382,1740){\makebox(0,0)[l]{\strut{}$\|f-S_{\Psi_0(m)}^{\bm X}\bm y\|_{L_2}^{2}+\sigma^2$}}\csname LTb\endcsname \put(5382,1381){\makebox(0,0)[l]{\strut{}$\FCV(S_{\Psi_1(m)}^{\bm X}\bm y)$}}\csname LTb\endcsname \put(5382,1021){\makebox(0,0)[l]{\strut{}$\|f-S_{\Psi_1(m)}^{\bm X}\bm y\|_{L_2}^{2}+\sigma^2$}}\csname LTb\endcsname \put(5382,662){\makebox(0,0)[l]{\strut{}$\FCV(S_{\Psi_2(m)}^{\bm X}\bm y)$}}\csname LTb\endcsname \put(5382,302){\makebox(0,0)[l]{\strut{}$\|f-S_{\Psi_2(m)}^{\bm X}\bm y\|_{L_2}^{2}+\sigma^2$}}\csname LTb\endcsname \put(3245,239){\makebox(0,0){\strut{}$m$}}}\gplbacktext
    \put(0,0){\includegraphics[width={396.00bp},height={127.00bp}]{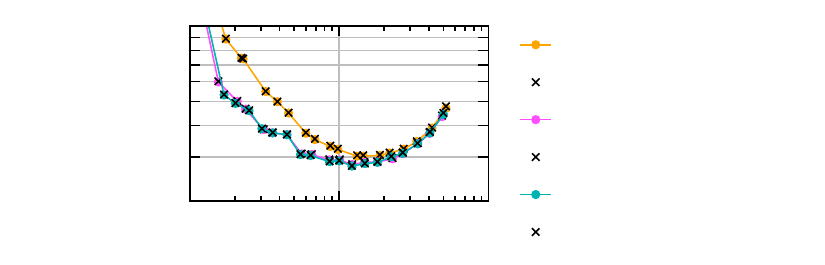}}\gplfronttext
  \end{picture}\endgroup
     \caption{Cross validation and $L_2$ error for the $d=10$ example with Gaussian noise.}\label{s10_cv}
\end{figure} The theoretical expectation of an improved convergence rate and smaller $L_2$ error is observed.
The second iteration only gives marginal gains.

 \section{Concluding remarks}\label{sec:conclusion} 

In this paper we considered the hyperparameter selection problem in that we select the shape of frequencies for the least squares ANOVA approximation based on function samples.
We set dozens of parameters based on the estimated smoothness properties of the function at hand, which we approximate from the Fourier coefficients of our approximation.

Whereas previous works for approximation \cite{SUV21} used linear information in the form of wavelet coefficients, we only relied on given samples, which is novel to the best of our knowledge.
We utilized the well-established, fast, and memory-efficient least squares ANOVA approximation, which is a linear method.
The hyperparameter tuning introduces nonlinearity in the method, which gains approximation quality without deteriorating efficiency.

Although we do not yet provide a self-contained theoretical guarantee for the entire procedure, each component of the method is supported by the theory presented in this work.
In particular, the smoothness estimation relies on a steady decay of the Fourier coefficients, which is a strong assumption and needs further investigation.
The numerical experiments show the advantage and reliability of the method. 
\section*{Acknowledgment}

The authors would like to thank Daniel Potts for the fruitful discussions and valuable suggestions during the preparation of this work.
Further, Felix~Bartel acknowledges the time spent in Sydney with funding from the ``High dimensional approximation, learning, and uncertainty'' ARC discovery project.

\printbibliography

\end{document}